\documentclass[12pt,reqno,a4paper]{amsart}            % AMS article style

\usepackage[totalwidth=450pt, totalheight=652pt]{geometry}

\usepackage{scrpage2}    % change headings, see Franzis book, p. 500)
\usepackage{scrtime}     % give time by \thistime

\usepackage{ifthen}       % used for ``draft'' mode

  \automark[subsection]{section}

\usepackage{graphicx}                % for graphics
\usepackage{amsmath, amssymb}%, theorem}
\usepackage{amsthm}

\usepackage{amscd}

\usepackage{accents}     % for \accentset, \ring...
\usepackage{yhmath}     % for \wideparen, ... (wide parenthesis over symbols)

\usepackage{mathtools}  % for \coloneqq etc.

\usepackage{bbm}         % blackboard (see: /usr/share/texmf/mytex/bbm.dvi)

\usepackage{mathrsfs}    % use ``real'' caligraphic letters for \mathcal
\renewcommand \mathcal \mathscr

\numberwithin{equation}{section}

\newcounter{myenumi}

\newcommand{\itemref}[1]{\eqref{#1}}

\newcommand{\myfont}{\sffamily}
\newtheoremstyle{mythmstyle}% name
  {\topsep}% Space above
  {\topsep}% Space below
  {\itshape}% Body font
  {}% Indent amount
  {\bfseries \myfont}% Theorem head font
  {.}%Punctuation after theorem head
  {.5em}%Space after theorem head
  {}% theorem head spec

\newtheoremstyle{mydefstyle}% name
  {\topsep}% Space above
  {\topsep}% Space below
  {\normalfont}% Body font
  {}% Indent amount
  {\bfseries \myfont}% Theorem head font
  {.}%Punctuation after theorem head
  {.5em}%Space after theorem head
  {}% theorem head spec

\swapnumbers                    % 1.1 Theorem instead of Theorem 1.1

\theoremstyle{mythmstyle}       % my style (new fonts) -- body italics

\newtheorem{theorem}{Theorem}[section]

\newtheorem{proposition}[theorem]{Proposition}
\newtheorem{lemma}[theorem]{Lemma}
\newtheorem{corollary}[theorem]{Corollary}

\newcounter{intro}

\theoremstyle{mydefstyle}        % my style (new fonts) -- body roman
\newtheorem{definition}[theorem]{Definition}
\newtheorem{assumption}[theorem]{Assumption}

\newtheorem{remark}[theorem]{Remark}
\newtheorem*{remark*}{Remark}

\expandafter\let\expandafter\oldproof\csname\string\proof\endcsname
\let\oldendproof\endproof
\renewenvironment{proof}[1][\bfseries\myfont\proofname]{%
  \oldproof[\bfseries \myfont #1]%
}{\oldendproof}

\makeatletter
\renewcommand\section{\@startsection{section}{1}%
  \z@{.7\linespacing\@plus\linespacing}{.5\linespacing}%
  {\Large\myfont\bfseries}}
\renewcommand\subsection{\@startsection{subsection}{2}%
  \z@{-.5\linespacing\@plus-.7\linespacing}{.5\linespacing}%
  {\large\myfont\bfseries}}
\renewcommand\subsubsection{\@startsection{subsubsection}{3}%
  \z@{.5\linespacing\@plus.7\linespacing}{-.5em}%
  {\myfont\bfseries}}
\renewenvironment{abstract}{%
  \ifx\maketitle\relax
    \ClassWarning{\@classname}{Abstract should precede
      \protect\maketitle\space in AMS document classes; reported}%
  \fi
  \global\setbox\abstractbox=\vtop \bgroup
    \normalfont\Small
    \list{}{\labelwidth\z@
      \leftmargin3pc \rightmargin\leftmargin
      \listparindent\normalparindent \itemindent\z@
      \parsep\z@ \@plus\p@
      
    }%
    \item[\hskip\labelsep%\scshape
      \myfont\bfseries
    \abstractname.]%
}{%
  \endlist\egroup
  \ifx\@setabstract\relax \@setabstracta \fi
}
\renewcommand\contentsnamefont{\myfont\bfseries}%{\scshape}
\renewcommand\@starttoc[2]{\begingroup
  \setTrue{#1}%
  \par\removelastskip\vskip\z@skip
  \@startsection{}\@M\z@{\linespacing\@plus\linespacing}%
    {.5\linespacing}{%\centering
      \contentsnamefont}{#2}%
  \ifx\contentsname#2%
  \else \addcontentsline{toc}{section}{#2}\fi
  \makeatletter
  \@input{\jobname.#1}%
  \if@filesw
    \@xp\newwrite\csname tf@#1\endcsname
    \immediate\@xp\openout\csname tf@#1\endcsname \jobname.#1\relax
  \fi
  \global\@nobreakfalse \endgroup
  \addvspace{32\p@\@plus14\p@}%
  \let\tableofcontents\relax
}
\renewcommand\@settitle{\begin{center}%
  \baselineskip14\p@\relax
    \LARGE
    \bfseries
    \myfont
  \@title
  \end{center}%
}
\renewcommand\@setauthors{%
  \begingroup
  \def\thanks{\protect\thanks@warning}%
  \trivlist
  \centering\footnotesize \@topsep30\p@\relax
  \advance\@topsep by -\baselineskip
  \item\relax
  \author@andify\authors
  \def\\{\protect\linebreak}%
  \large
  \myfont\bfseries\authors
  \ifx\@empty\contribs
  \else
    ,\penalty-3 \space \@setcontribs
    \@closetoccontribs
  \fi
  \endtrivlist
  \normalfont\myfont\@setaddresses
  \endgroup
}
\renewcommand\@setaddresses{\par
  \nobreak \begingroup
\footnotesize
  \def\author##1{\nobreak\addvspace\bigskipamount}%
  \def\\{\unskip, \ignorespaces}%
  \interlinepenalty\@M
  \def\address##1##2{\begingroup
    \par\addvspace\bigskipamount\indent
    \@ifnotempty{##1}{(\ignorespaces##1\unskip) }%
    {%\scshape
      \ignorespaces##2}\par\endgroup}%
  \def\curraddr##1##2{\begingroup
    \@ifnotempty{##2}{\nobreak\indent\curraddrname
      \@ifnotempty{##1}{, \ignorespaces##1\unskip}\/:\space
      ##2\par}\endgroup}%
  \def\email##1##2{\begingroup
    \@ifnotempty{##2}{\nobreak\indent\emailaddrname
      \@ifnotempty{##1}{, \ignorespaces##1\unskip}\/:\space
      \ttfamily##2\par}\endgroup}%
  \def\urladdr##1##2{\begingroup
    \def~{\char`\~}%
    \@ifnotempty{##2}{\nobreak\indent\urladdrname
      \@ifnotempty{##1}{, \ignorespaces##1\unskip}\/:\space
      \ttfamily##2\par}\endgroup}%
  \addresses
  \endgroup
}
\renewcommand\enddoc@text{\ifx\@empty\@translators \else\@settranslators\fi
}
\renewcommand\@secnumfont{\myfont\bfseries} %{\mdseries}
\makeatother

\newcommand{\Sec}[1]{Section~\ref{sec:#1}}

\newcommand{\Subsec}[1]{Subsection~\ref{ssec:#1}}

\newcommand{\Thm}[1]{Theorem~\ref{thm:#1}}
\newcommand{\Thms}[2]{Theorems~\ref{thm:#1} and~\ref{thm:#2}}
\newcommand{\Thmss}[3]{Theorems~\ref{thm:#1},~\ref{thm:#2} and~\ref{thm:#3}}

\newcommand{\Thmenum}[2]{Theorem~\ref{thm:#1}~(\ref{#2})}

\newcommand{\Lem}[1]{Lemma~\ref{lem:#1}}

\newcommand{\Cor}[1]{Corollary~\ref{cor:#1}}

\newcommand{\Corenum}[2]{Corollary~\ref{cor:#1}~(\ref{#2})}

\newcommand{\Prp}[1]{Proposition~\ref{prp:#1}}
\newcommand{\Prps}[2]{Propositions~\ref{prp:#1} and~\ref{prp:#2}}

\newcommand{\Rem}[1]{Remark~\ref{rem:#1}}

\newcommand{\Def}[1]{Definition~\ref{def:#1}}

\newcommand{\Defenum}[2]{Definition~\ref{def:#1}~(\ref{#2})}
\newcommand{\Defenums}[3]{Definition~\ref{def:#1}~(\ref{#2}) and~(\ref{#3})}

\newcommand{\Ass}[1]{Assumption~\ref{ass:#1}}

\newcommand{\abs}[2][{}]{\lvert{#2}\rvert_{{#1}}}    % abs value
\newcommand{\abssqr}[2][{}]{\lvert{#2}\rvert^2_{#1}} % abs squared
\newcommand{\bigabs}[2][{}]{\bigl\lvert{#2}\bigr\rvert_{#1}}     % abs
\newcommand{\normsymb}{\|}
\newcommand{\bignormsymb}[1]{#1\|}

\newcommand{\norm}[2][{}]{\normsymb{#2}\normsymb_{{#1}}}    % norm
\newcommand{\normsqr}[2][{}]{\normsymb{#2}\normsymb^2_{#1}} % norm squared
\newcommand{\bignorm}[2][{}]{\bignormsymb{\bigl}{#2}\bignormsymb{\bigr}_{#1}}
\newcommand{\iprod}[3][{}]{\langle{#2},{#3}\rangle_{#1}}  % inner product
\newcommand{\bigiprod}[3][{}]{\bigl\langle{#2},{#3}\bigr\rangle_{#1}}

\newcommand{\set}[2]{\{ \, #1 \, | \, #2 \, \} }      % set { #1 | #2 }
\newcommand{\bigset}[2]{\bigl\{ \, #1 \, \bigl|\bigr. \, #2 \, \bigr\} }

\newcommand{\BIGset}[2]{\left\{ \, #1 \, \Bigl|\Bigr. \, #2 \, \right\} }

\DeclareMathOperator*{\bigdcup}{\mathaccent\cdot{\bigcup}}

\newcommand{\map}[3]{ #1 \colon #2 \longrightarrow #3}    % maps
\newcommand{\embmap}[3]{ #1 \colon #2 \hookrightarrow #3} % embedding map

\newcommand{\bd}  {\partial}          % symbol for boundary of a set
\newcommand{\clo}[2][]{\overline{{#2}}^{#1}} %  symbol for closure
\newcommand{\intr}[1]{\ring{{#1}}}    %  symbol for interior

\newcommand{\restr}[1]{{\restriction}_{#1}} % symbol for map restriction

\newcommand{\card}[1]{\lvert#1\rvert}   % from AMS proceedings file
\newcommand{\dd}    {\, \mathrm d}    % not optimal: no \, if at beginning

\DeclareMathOperator{\dom}    {dom}

\DeclareMathOperator{\id}     {id}   % identity map
\DeclareMathOperator{\Ric}    {Ric}

\DeclareMathOperator{\supp}   {supp}
\DeclareMathOperator{\vol}    {vol}

\newcommand{\specsymb} {\sigma} % symbol for spectrum

\newcommand{\spec}[2][{}]   {\specsymb_{\mathrm{#1}}(#2)}

\newcommand{\eps}{\varepsilon} % shortcut
\renewcommand{\phi}{\varphi}   % shortcut
\renewcommand{\rho}{\varrho}   % shortcut
\newcommand{\R}{\mathbb{R}} % symbol for real numbers
\newcommand{\C}{\mathbb{C}} % symbol for complex numbers
\newcommand{\N}{\mathbb{N}} % symbol for natural numbers
\newcommand{\Sphere}{\mathbb{S}} % symbol for sphere
\newcommand{\1}{\mathbbm 1}                    % blackboard 1

\newcommand{\e}{\mathrm e}  %Euler number
\DeclareMathSymbol{\widetildesym}{\mathord}{largesymbols}{"65}
\newcommand\lowerwidetildesym[1][-1.3ex]{%
  \text{\smash{\raisebox{#1}{%
    $\widetildesym$}}}}
\newcommand\fixwidehat[1]{%
  \mathchoice
    {\accentset{\displaystyle\lowerwidetildesym[-1.2ex]}{#1}}
    {\accentset{\textstyle\lowerwidetildesym[-1.2ex]}{#1}}
    {\accentset{\scriptstyle\lowerwidetildesym[-1.9ex]}{#1}}
    {\accentset{\scriptscriptstyle\lowerwidetildesym[-2.5ex]}{#1}}
}
\newcommand{\wt}{\fixwidehat}           % shortcut

\newcommand {\qf}[1]{\mathfrak{#1}}    % font for quadratic forms

\newcommand{\HS}{\mathcal H}           % symbol for Hilbert space
\newcommand{\Sobsymb} {\mathsf H} % symbol for Sobolev space
\newcommand{\Sobnsymb} {\ring{\mathsf H}} % symbol for Sobolev space
\newcommand{\Contsymb} {\mathsf C}     % symbol for cont. space
\newcommand{\Lsymb}    {\mathsf L}     % symbol for int L-spaces
\newcommand{\Sobspace}[1][1]{\Sobsymb^{#1}} 
\newcommand{\Sobnspace}[1][1]{\Sobnsymb^{#1}}

\newcommand{\Contspace}[1][{}]{\Contsymb^{#1}}     % symbol for cont. space
\newcommand{\Lpspace}[1][p]    {\Lsymb_{#1}}     % symbol for int L-spaces
\newcommand{\Lsqrspace}    {\Lpspace[2]}     % symbol for int L-spaces
\newcommand{\Ci} [2][{}]{\Contspace [\infty]_{#1} ({#2})}
\newcommand{\Cci}[1]{\Ci[\mathrm c]{#1}}%{\Contsymb^\infty_{\mathrm c} ({#1})}
\newcommand{\Lp}[2][p]{\Lpspace [#1]({#2})} % L_#1(#2)-spaces

\newcommand{\Lsqr}[2][{}]{\Lsqrspace^{#1}({#2})} % L_2^{#1}(#2)-spaces
\newcommand{\Sob}[2][1]{\Sobspace [#1]({#2})}         % Sobolev space
\newcommand{\Sobn}[2][1]{\Sobnspace [#1]({#2})}  % Sob space ^0

\newcommand{\Sobx}[3][1]{\Sobspace [#1]_{{#2}}({#3})} % Sobolev space
\newcommand{\Neu}{{\mathrm N}}              % symbol for Neumann bd cond
\newcommand{\Dir}{{\mathrm D}}              % symbol for Dirichlet bd cond
\newcommand{\laplacian}[2][{}]{\Delta_{{#2}}^{{#1}}} 
\newcommand{\laplacianD}[1]{\laplacian[\Dir]{#1}} % symb f Dir-Laplacian
\newcommand{\laplacianN}[1]{\laplacian[\Neu]{#1}} % symb f Neu-Laplacian
\newcommand{\Err}{\mathrm O}

\newcommand {\loc}{\mathrm{loc}}

\newcommand{\spacetext}[2]{\hspace*{#1}\text{#2}\hspace*{#1}}
\newcommand{\quadtext}[1]{\spacetext{1em}{#1}}% \quad is 1em
\newcommand{\qquadtext}[1]{\spacetext{2em}{#1}} %\qquad is 2em

\newcommand{\eucl}{\mathrm{eucl}}

\newcommand{\CSob}{C_{\mathrm{Sob}}}

\newcommand{\Cext}{C_{\mathrm{ext}}}
\newcommand{\Cellreg}{C_{\mathrm{ell.reg}}}
\newcommand{\cellreg}{c_{\mathrm{ell.reg}}}

\providecommand{\no}{n$^{\textrm o}$ }
\newcommand{\cref}[1]{Corollary~\ref{#1}}

\newcommand{\eref}{\eqref}

\newcommand{\sref}[1]{Section~\ref{#1}}
\newcommand{\tir}{\discretionary{.}{}{---\kern.7em}}

\renewcommand{\Sphere}[1]{\mathbb S^{#1}}   %m-1

\newcommand\mcD{\mathcal D}

\begin{document}

\title[Wildly perturbed manifolds] 
    {Wildly perturbed manifolds: norm resolvent and spectral convergence}

\author{Colette Ann\'e}% 
\address{Laboratoire de Math\'ematiques Jean
  Leray, CNRS -- Universit\'e de Nantes, Facult\'e des Sciences, BP 92208,
  44322 Nantes, France}
\email{colette.anne@univ-nantes.fr}

\author{Olaf Post}
\address{Fachbereich 4 -- Mathematik,
  Universit\"at Trier,
  54286 Trier, Germany}
\email{olaf.post@uni-trier.de}

% \date{\today}}  % final version
\date{\today, \thistime,  \emph{File:} \texttt{\jobname.tex}} 
           % draft version

%-------------------------------------------------------------
% Abstract.
%-------------------------------------------------------------
\begin{abstract}
  Since the publication of the important work of Rauch and Taylor
  \cite{rauch-taylor:75} a lot has been done to analyse wild
  perturbations of the Laplace-Beltrami operator. Here we present
  results concerning the norm convergence of the resolvent.  We
  consider a (not necessarily compact) manifold with many small balls
  removed, the number of balls can increase as the radius is
  shrinking, the number of balls can also be infinite.  If the
  distance of the balls shrinks less fast than the radius, then we
  show that the Neumann Laplacian converges to the unperturbed
  Laplacian, i.e., the obstacles vanish.  In the Dirichlet case, we
  have two cases: if the balls are too sparse, the limit operator is
  again the unperturbed one, while if the balls concentrate at a
  certain region (they become ``solid'' in a region), the limit
  operator is the Dirichlet Laplacian on the complement outside the
  solid region.  Our work is based on a norm convergence result for
  operators acting in varying Hilbert spaces described in the
  book~\cite{post:12} by the second author. 
\end{abstract}


\maketitle

%\keywords{ }

%%%%%%%%%%%%%%%%%%%%%%%%%%%%%%%%%%%%

%-----------------------------------------------------------------------
%
% aaaa
\section{Introduction}
\label{sec:intro}
%
%-----------------------------------------------------------------------

Since the publication of the important work of Rauch and
Taylor~\cite{rauch-taylor:75} a lot have been done to analyse wild
perturbations of the Laplace-Beltrami operator. Wild perturbations
refers here to increase the complexity of topology.  In particular, we
show convergence of the Laplace-Beltrami operator on manifolds with an
increasing number of small holes. % or added handles.

%-----------------------------------------------------------------------
\subsection{Main results}
%-----------------------------------------------------------------------
In this article, we present results concerning the norm convergence of
the resolvent.  Since the perturbation changes the space on which the
operators act, we need to define a \emph{generalised norm resolvent
  convergence} for operators on varying spaces (see \Def{gnrs}).  This
powerful tool and many consequences (like convergence of eigenvalues,
eigenfunctions, functions of the operators such as spectral
projections, the heat operator etc.) is explained in detail in a book
by the second author~\cite{post:12}.  Let us stress here that we do
not need a compactness assumption on the space or the resolvents as in
most of the previous works (see \sref{previous}).  Moreover, the
abstract convergence result shows its full strengths especially when
the perturbed space is not a subset of the unperturbed one or vice
versa: an example is adding many small handles; we treat this problem
in a subsequent publication~\cite{anne-post:pre18b}.

We give sufficient conditions on the obstacles in
\Thms{neu-fading}{dir-fading} to have (generalised norm resolvent)
convergence to the unperturbed situation (obstacles without an effect)
where we remove a family of obstacles and consider on the remaining
manifold either the Neumann or Dirichlet Laplacian.  In the Dirichlet
case, there is a regime when the obstacles can become ``solid''
(\Thm{dir-solid}).  These abstract results use as assumption e.g.\
non-concentrating of energy-bounded functions on small parts and
extension properties in the Neumann case.

We make these abstract results concrete in
\Thmss{neu-fading.balls}{dir-fading.balls}{dir-solid.balls}, where we
assume that the obstacles consists of many small balls having a
certain minimal distance, and filling up the ``solid'' region for
\Thm{dir-solid.balls}, a terminology introduced
in~\cite{rauch-taylor:75} to describe the situation under the name
``crushed ice problem'' where small obstacles such as holes maintained
at zero temperature increase in number while their size converge to
$0$ in such a way that they \emph{freeze} at the limit.  A typical
assumption here is that small balls in the manifold look
everywhere roughly the same; this is assured if the harmonic radius is
uniformly positive; and the latter follows if the manifold has
\emph{bounded geometry}, see \Def{bdd.geo} and \Prp{eucl.metric}.

Let us first explain the main idea behind the abstract convergence
tool: In all our results, we deal with an $\eps$-dependent space
$X_\eps$ and suitable Laplace operators $\laplacian \eps$ acting on
$X_\eps$ for each $\eps \ge 0$.  We define a generalised norm
resolvent convergence for $\laplacian \eps$ to a limit Laplacian
$\laplacian 0$.  To do so, we need so-called \emph{identification} or
\emph{transplantation} operators $\map{J=J_\eps}{\Lsqr {X_0}} {\Lsqr
  {X_\eps}}$ and $\map{J'=J_\eps'}{\Lsqr {X_\eps}}{\Lsqr {X_0}}$,
which are asymptotically unitary (cf.~\eqref{eq:gnrs.a}) and
intertwine the resolvents (cf.~\eqref{eq:gnrs.b}) in the following
sense:
%-----------------------------------------------------------------------
\begin{definition}
  \label{def:gnrs}
  We say that $\laplacian \eps$ \emph{converges in general norm
    resolvent sense} to $\laplacian 0$ if there exist bounded
  operators $J$ and $J'$ and $m \ge 0$ such that
  \begin{subequations}
    \begin{align}
      \label{eq:gnrs.a}
      \norm{(\id_{\HS_0}-J'J)R_0}& \le \delta_\eps,&
      \norm{(\id_{\HS_\eps}-JJ')R_\eps}& \le \delta_\eps,\\
      \label{eq:gnrs.b}
      \norm{(JR_0-R_\eps J)R_0^{m/2}} &\le \delta_\eps,
    \end{align}
  \end{subequations}
  where $R_0:=(\laplacian 0 +1)^{-1}$ and $R_\eps:=(\laplacian \eps
  +1)^{-1}$ for $\eps > 0$ and where $\delta_\eps \to 0$ as $\eps
  \to 0$.
\end{definition}
%-----------------------------------------------------------------------
The name is justified as follows: if $\HS_\eps=\HS_0$, then
generalised norm resolvent convergence (with $m=0$) is just the
classical norm resolvent convergence if one chooses
$J=J'=\id_{\HS_0}$.  In \Sec{main.tool}, we interpret $\delta_\eps$ as
a sort of ``distance'' between $\laplacian 0$ and $\laplacian \eps$,
or more, precisely, between their corresponding quadratic forms $\qf
d_0$ and $\qf d_\eps$, and call such forms
\emph{$\delta_\eps$-quasi-unitarily equivalent}.  If this distance
converges to $0$, then $\laplacian \eps$ converges to $\laplacian 0$
in generalised norm resolvent convergence, see \Sec{main.tool}.

Once we have this generalised norm resolvent convergence, similar
conclusions as for the classical norm resolvent convergence are valid.
In particular, we have norm convergence (using also $J$ and $J'$) of
the corresponding functional calculus, i.e., of $\phi(\Delta_\eps)$
towards $\phi(\Delta_0)$ for suitable functions $\phi$ such as
$\phi=\1_{[a,b]}$ with $a,b \notin \spec {\Delta_0}$ (spectral
projections) or $\phi(\lambda)=\e^{-t\lambda}$ (heat operator), see
\Thm{quasi-uni}.  Moreover, we conclude the following spectral
convergence:
%-----------------------------------------------------------------------
\begin{theorem}[{\cite[Thms.~4.3.3--4.3.5]{post:12}}]
  \label{thm:spectrum}
  Assume that $\Delta_\eps$ converges to $\Delta_0$ in generalised
  norm resolvent sense then
  \begin{equation*}
    \spec[\bullet] {\Delta_\eps} \to \spec[\bullet] {\Delta_0}
  \end{equation*}
  uniformly (i.e., in Hausdorff distance) on any compact interval
  $[0,\Lambda]$.  Here, $\spec[\bullet] {\Delta_\eps}$ stands for the
  entire spectrum or the essential spectrum of $\Delta_\eps$ for $\eps
  \ge 0$.

  If $\lambda_0 \in \spec[disc] {\Delta_0}$ is an eigenvalue of
  multiplicity $\mu>0$, then there exist $\mu$ eigenvalues (not
  necessarily all distinct) $\lambda_{\eps,j}$, $j=1 \dots \mu$, such
  that $\lambda_{\eps,j} \to \lambda_0$ as $\eps \to 0$.  In particular,
  if $\mu=1$ and if $\psi_0 \in \HS_0$ is the corresponding normalised
  eigenvector, then there exists a family of normalised eigenvectors
  $\psi_\eps$ of $\Delta_\eps$ such that
  \begin{equation}
    \label{eq:conv.ef}
    \norm{J \psi_0 - \psi_\eps} \to 0
    \quadtext{and}
    \norm{J' \psi_\eps - \psi_0} \to 0
  \end{equation}
  as $\eps \to 0$.
\end{theorem}
%-----------------------------------------------------------------------

% the
% following results (this follows from the abstract setting explained in
% detail in \Sec{main.tool}):
% %-----------------------------------------------------------------------
% \begin{theorem}
%   \label{thm:consequences.gnrs}
%   Assume that $\laplacian \eps$ converges in general norm
%     resolvent sense to $\laplacian 0$.  Then the following holds:
%   \begin{gather*}
%     \norm{(J \phi(\laplacian \eps) - \phi(\laplacian 0) J)R_0^{m/2}} 
%     = \err(1)\\
%     \norm{(J \1_{[a,b]}(\laplacian \eps)-\1_{[a,b]}(\laplacian 0) J) R_0^{m/2}} 
%     = \err(1)
%   \end{gather*}
%   where $\phi$ is a measurable function such that $\phi$ is continuous
%   on a neighbourhood $U$ of the spectrum $\spec {\laplacian 0}$ of 
%   $\laplacian 0$; moreover, $a,b
%   \notin \spec{\laplacian 0}\setminus \{0\}$, and the error terms
%   depend only on $\phi$ resp.\ $a$, $b$ and $U$.
% \end{theorem}
%-----------------------------------------------------------------------
% We also have convergence of the spectra as stated in \Thm{spectrum}
% for sequences of quadratic forms.
% %-----------------------------------------------------------------------

%-----------------------------------------------------------------------
\subsection{Previous works}\label{previous}
%-----------------------------------------------------------------------
The results of Rauch and Taylor inspired a lot of works (74 items in
MathSciNet), mostly concerning the convergence of eigenvalues.  We
mention here three points.

The asymptotic behaviour of Neumann eigenvalues was studies for a
single hole for bounded domains or compact manifolds in
\cite{ozawa:83,hempel:06, lanza:12} and the Dirichlet eigenvalues
in~\cite{chavel-feldman:78, courtois:95} where we find precise
estimates; it applies also to the $\eps$-neighbourhood of compact
subset, see also~\cite{chavel-feldman:88} for the calculation of the
first correction term.

Daners~\cite{daners:03} considers the \emph{norm} convergence of
resolvents of Dirichlet Laplacians for perturbation of Euclidean
\emph{bounded} domains (or at least those with compact resolvent), the
norm convergence follows from the strong one under the assumption of
compactness of the limit resolvent, see also~\cite{daners:08} for a
survey and the references therein.  Our approach is more general as it
does not a priori assume that the perturbed and unperturbed domains
are embedded in a common space as in~\cite{daners:03,daners:08}.
Moreover, we obtain explicit error estimates in terms of
$\delta_\eps$.  For an older survey about strong resolvent convergence
and perturbations of Euclidean domains, we refer to~\cite{henrot:94}.

Finally, convergence of resolvents has also been studied via the
\emph{homogenisation} point of view, mainly on bounded Euclidean
domains or compact manifolds: In~\cite{balzano-notarantonio:98}
Balzano and Notarantonio consider a compact Riemannian manifold with
an increasing finite number of small balls removed.  They show that if
the balls are placed randomly and if their capacity converges, then
the Dirichlet Laplacian on the manifold less the holes converges in
\emph{strong} resolvent sense to a Laplacian plus a potential given by
the random distribution of ball centres.  The proof is based on
earlier works of Balzano ~\cite{balzano:88} using
$\Gamma$-convergence, see~\cite{dal-maso:93}.  More recent works can
be found in~\cite{khrabustovskyi:09} or~\cite{khrabustovskyi:13}.  For
a similar approach using the above mentioned generalised norm
resolvent convergence in the homogenisation case, we refer
to~\cite{khrabustovskyi-post:pre17} and the references cited therein.
For an approach using the already shown strong resolvent convergence
to show norm resolvent convergence (similarly as
in~\cite{daners:03,daners:08}, but even for general unbounded domains)
we refer to~\cite{dondl-cherednichenko-roesler:pre17}.

%-----------------------------------------------------------------------
\subsection{Structure of the article}
%-----------------------------------------------------------------------
In \Sec{main.tool} we briefly describe the main tool of norm
convergence of operators on varying Hilbert spaces.  In \Sec{mfd-lapl}
we briefly introduce Laplacians and Sobolev spaces on manifolds, the
harmonic radius manifolds of bounded geometry.  Moreover, introduce
the concept of non-concentration in \Def{non-concentr} and
\Prp{non-concentr2}.

In \Sec{neu-fading} we present the situation for obstacles with
Neumann boundary condition, the main result \Thm{neu-fading} for
abstract fading obstacles, and in \Thm{neu-fading.balls} where each
obstacle is is a disjoint union of many small balls of radius $\eps$.
Similarly \Sec{dir-fading} contains results for fading Dirichlet
obstacles and many balls in \Thms{dir-fading}{dir-fading.balls}.
Finally, \Sec{dir-solid} is about Dirichlet obstacles that become
``solid'', again an abstract version and one for many balls removed in
\Thms{dir-solid}{dir-solid.balls}.  We conclude with an appendix,
where we collect some additional facts about estimates on manifolds.

%-----------------------------------------------------------------------
\subsection*{Acknowledgements}
%-----------------------------------------------------------------------
OP would like to thank  the F\'ed\'eration Recherche Math\'ematiques des Pays 
de Loire for the hospitality, through its program \emph{G\'eanpyl}, at the 
\emph{Universit\'e de Nantes}.

%
%-----------------------------------------------------------------------
%
% bbbb
\section{Main tool: norm convergence of operators on varying Hilbert
  spaces}
\label{sec:main.tool}
%
%-----------------------------------------------------------------------
The second author of the present article proposed in~\cite{post:06}
and in more detail in the monograph~\cite{post:12} a general
framework which assures a \emph{generalised} norm resolvent
convergence for operators $\Delta_\eps$ converging to $\Delta_0$ as
$\eps \to 0$.  Here, each operator $\Delta_\eps$ acts in a Hilbert
space $\HS_\eps$ for $\eps \ge 0$; and the Hilbert spaces are allowed
to depend on $\eps$.  In typical applications, the Hilbert spaces
$\HS_\eps$ are of the form $\Lsqr{X_\eps}$ for some metric measure
space $X_\eps$ which is considered as a perturbation of a ``limit''
metric measure space $X_0$; and typically, there is a topological
transition between $\eps>0$ and $\eps=0$.

In order to define the convergence, we define a sort of ``distance''
$\delta_\eps$ between $\wt \Delta\coloneqq\Delta_\eps$ and
$\Delta\coloneqq\Delta_0$, in the sense that if $\delta_\eps \to 0$ then
$\Delta_\eps$ converges to $\Delta_0$ in the above-mentioned
generalised norm resolvent sense.

Let $\HS$ and $\wt \HS$ be two separable Hilbert spaces.  We say that
$(\qf d,\HS^1)$ is an \emph{energy form in $\HS$} if $\qf d$ is a
closed, non-negative quadratic form in $\HS$ with domain $\HS^1$, i.e., if 
$\qf d(f)\coloneqq\qf d(f,f)\ge 0$ for some sesquilinear form 
$\map {\qf d}{\HS^1 \times \HS^1}\C$, denoted by the same symbol, with
$\HS^1=:\dom \qf d$ endowed with the norm defined by
\begin{equation}
  \label{eq:qf.norm}
  \normsqr[1] f
  \coloneqq \normsqr[\HS^1] f
  \coloneqq \normsqr[\HS] f + \qf d(f),
\end{equation}
so $\HS^1$ is itself a Hilbert space and a dense set in $\HS$.  We
denote $\Delta$ the corresponding non-negative, self-adjoint operator
the \emph{energy operator} associated with $(\qf d,\HS^1)$ (see
e.g.~\cite[Sec.~VI.2]{kato:66}).  Similarly, let $(\wt {\qf d},\wt
\HS^1)$ be an energy form in $\wt \HS$ with energy operator $\wt
\Delta$.

Associated with an energy operator $\Delta$, we can define a natural
\emph{scale of Hilbert spaces} $\HS^k$ defined via the \emph{abstract
  Sobolev norms}
\begin{equation}
  \label{eq:def.abstr.sob.norm}
  \norm[\HS^k] f
  \coloneqq \norm[k] f 
  \coloneqq \norm{(\Delta+1)^{k/2}f}.
\end{equation}
Then $\HS^k=\dom \Delta^{k/2}$ if $k \ge 0$ and $\HS^k$ is the
completion of $\HS$ with respect to the norm $\norm[k] \cdot$ for
$k<0$.  
Obviously, the scale of Hilbert spaces for $k=1$ and its
associated norm agrees with $\HS^1$ and $\norm[1]\cdot$ defined above
(see~\cite[Sec.~3.2]{post:12} for details).  Similarly, we denote by
$\wt \HS^k$ the scale of Hilbert spaces associated with $\wt \Delta$.

We denote by $\spec \Delta$ the spectrum of the energy operator and by
$R(z)=(\Delta-z)^{-1}$ its resolvent at $z\in\C \setminus \spec
\Delta)$ and for short $R=R(-1)=(\Delta+1)^{-1}$, we use similar
notations for $\wt \Delta$.

We now need pairs of so-called \emph{identification} or
\emph{transplantation operators} acting on the Hilbert spaces and
later also pairs of identification operators acting on the form
domains.  Note that our definition is slightly more general than the
ones in~\cite[Secs. 4.1, 4.2 and 4.4]{post:12}.  The new point here is
that we allow the (somehow ``smoothing'') resolvent power of order
$k/2$ on the right hand side in~\eqref{eq:quasi-uni.d'} also for
$k>0$.
%-----------------------------------------------------------------------
\begin{definition}
  \label{def:quasi-uni}
  \begin{subequations}
    \label{eq:quasi-uni}
    Let $\delta \ge 0$, and let $\map J \HS {\wt \HS}$ and $\map {J'}
    {\wt \HS}\HS$ be linear bounded operators.  Moreover, let $\delta
    \ge 0$, and let $\map {J^1} {\HS^1} {\wt \HS^1}$ and $\map
    {J^{\prime1}} {\wt \HS^1}{\HS^1}$ be linear bounded operator
    on the energy form domains.
    \begin{enumerate}
    \item We say that $J$ is \emph{$\delta$-quasi-unitary} with
      \emph{$\delta$-quasi-adjoint} $J'$ if
      \begin{gather}
        \label{eq:quasi-uni.a}
        \norm{Jf}\le (1+\delta) \norm f, \quad
        \bigabs{\iprod {J f} u - \iprod f {J' u}}
        \le \delta \norm f \norm u
        \qquad (f \in \HS, u \in \wt \HS),\\
        \label{eq:quasi-uni.b}
        \norm{f - J'Jf}
        \le \delta \norm[1] f, \quad
        \norm{u - J'Ju}
        \le \delta \norm[1] u \qquad (f \in \HS^1, u \in \wt \HS^1).
      \end{gather}
      
    \item We say that $J^1$ and $J^{\prime1}$ are \emph{$\delta$-compatible}
      with the identification operators $J$ and $J'$ if
      \begin{equation}
        \label{eq:quasi-uni.c}
        \norm{J^1f - Jf}\le \delta \norm[1]f, \quad
        \norm{J^{\prime1}u - J'u} \le \delta \norm[1] u
        \qquad (f \in \HS^1, u \in \wt \HS^1).
      \end{equation}
      
    \item We say that the energy forms $\qf d$ and $\wt {\qf d}$ are
      \emph{$\delta$-close (of order $k \ge 1$)} if
      \begin{equation}
        \label{eq:quasi-uni.d}
        \bigabs{\wt{\qf d}(J^1f, u) - \qf d(f, J^{\prime1}u)} 
        \le \delta \norm[k] f \norm[1] u
        \qquad (f \in \HS^k, u \in \wt \HS^1).
      \end{equation}
      
    \item We say that $\qf d$ and $\wt {\qf d}$ are
      \emph{$\delta$-quasi unitarily equivalent (of order $k \ge 1$)},
      if~\eqref{eq:quasi-uni.a}--\eqref{eq:quasi-uni.d} are fulfilled,
      i.e.,
      \begin{itemize}
      \item if there exists identification operators $J$ and $J'$ such
        that $J$ is $\delta$-quasi unitary with $\delta$-adjoint $J'$
        (i.e., \eqref{eq:quasi-uni.a}--\eqref{eq:quasi-uni.b} hold);
      \item if there exists identification operators $J^1$ and
        $J^{\prime1}$ which are $\delta$-compatible with $J$ and $J'$
        (i.e., \eqref{eq:quasi-uni.c} holds);
      \item and if $\qf d$ and $\wt {\qf d}$ are $\delta$-close (of
        order $k$) (i.e., \eqref{eq:quasi-uni.d} holds).
      \end{itemize}
    \end{enumerate}
  \end{subequations}
\end{definition}
%-----------------------------------------------------------------------

%-----------------------------------------------------------------------
In operator norm notation, $\delta$-quasi unitary equivalence means
\begin{gather}
  \label{eq:quasi-uni.a'}
  \tag{\ref{eq:quasi-uni.a}'}
  \norm J \le 1+\delta, \qquad \norm{J^* - J'} \le \delta\\
  \label{eq:quasi-uni.b'}
  \tag{\ref{eq:quasi-uni.b}'}
  \norm{(\id_\HS - J'J)R^{1/2}} \le \delta, \qquad
  \norm{(\id_{\wt \HS} - J J')\wt R^{1/2}} \le \delta,\\
  \label{eq:quasi-uni.c'}
  \tag{\ref{eq:quasi-uni.c}'}
  \norm{(J^1-J)R^{1/2}} \le \delta, \qquad
  \norm{(J^{\prime1}-J')\wt R^{1/2}} \le \delta,\\
  \label{eq:quasi-uni.d'}
  \tag{\ref{eq:quasi-uni.d}'}
  \norm{\wt R^{1/2}(\wt \Delta J^1- (J^{\prime1})^* \Delta)R^{k/2}} \le \delta,
\end{gather}
where $R\coloneqq(\Delta+1)^{-1}$ resp.\ $\wt R\coloneqq(\wt
\Delta+1)^{-1}$ denotes the resolvent of $\Delta$ resp.\ $\wt \Delta$
in $-1$.  Moreover, $\map{(J^{\prime1})^*}{\HS^{-1}}{\wt \HS^{-1}}$
where $(\cdot)^*$ denotes here the dual map with respect to the dual
pairing $\HS^1 \times \HS^{-1}$ induced by the inner product on $\HS$
and similarly on $\wt \HS$.  Moreover, $\Delta$ is interpreted as
$\map \Delta {\HS^1} {\HS^{-1}}$, and similarly for $\wt \Delta$.

To give a flavour of the ideas, we give a short proof of the following
result:
%-----------------------------------------------------------------------
\begin{proposition}
  \label{prp:quasi-uni}
  Let $\qf d$ and $\wt {\qf d}$ be $\delta$-quasi unitarily equivalent
  (of order $k \ge 1$), then the following holds true:
  \begin{equation}
    \label{eq:prp.quasi-uni}
    \bignorm{\bigl(J R - \wt R J\bigr)R^{(k-2)/2}}
    \le 7\delta.
  \end{equation}
  In particular, if the energy forms $\qf d_\eps$ and $\qf d_0$ are
  $\delta_\eps$-quasi-unitarily equivalent of order $k \ge 1$ then the
  corresponding operators $\Delta_\eps$ converge in generalised norm
  resolvent sense to $\Delta_0$ of order $m=\max\{k-2,0\}$ and the
  conclusions of \Thm{spectrum} hold.
\end{proposition}
%-----------------------------------------------------------------------
\begin{proof}
  We have the expansion
  \begin{multline*}
    (J R - \wt R J)R^{(k-2)/2}
    = (J - J^1)R^{k/2}
    + \bigl(J^1 R-\wt R(J^{\prime1})^*\bigr)R^{(k-2)/2}\\
    + \wt R^{1/2}\bigl(\wt R^{1/2}((J^{\prime1})^*-(J')^*)\bigr)R^{(k-2)/2}
    + \wt R \bigl((J')^*-J\bigr)R^{(k-2)/2},
  \end{multline*}
  where the second term can be further expanded into
  \begin{align}
    \nonumber
    \bigl(J^1 R-\wt R(J^{\prime1})^*\bigr)R^{(k-2)/2}
    &=
       \wt R
       \bigl(
         (\wt \Delta + 1) J^1
         - (J^{\prime1})^*(\Delta + 1)
       \bigr)
       R^{k/2}\\
    \nonumber
    &=\wt R (\wt \Delta J^1 - (J^{\prime1})^*\Delta) R^{k/2}\\
    &\qquad+ \wt R
       \bigl(
          (J^1 - J) + (J - (J')^*) + ((J')^* - (J^{\prime1})^*)
       \bigr) 
       R^{k/2}.
    \label{eq:proof.quasi-uni}
   \end{align}
  Taking the operator norm, and using $\norm {A^*}=\norm A$ for dual
  of an operator, we obtain from the last two equations
  \begin{multline*}
    \norm{(J R - \wt R J)R^{(k-2)/2}}
    \le 2\norm{(J - J^1)R^{1/2}}
       +\norm{\wt R^{1/2}(\wt \Delta J^1 - (J^{\prime1})^*\Delta)R^{k/2}}\\
%       \hspace*{0.1\textwidth}
       +2\norm{(J^{\prime1}-J')\wt R^{1/2}}
       +2 \norm{J'-J^*}
       \le 7\delta. \qedhere
  \end{multline*}
\end{proof}
%-----------------------------------------------------------------------

%-----------------------------------------------------------------------
\begin{remark}
  \label{rem:quasi-uni}
  The last two items explain the notation in two extreme cases:
  \begin{enumerate}
  \item If $k \in \{1,2\}$ then we can ignore the factors
    $R_0^{(k-2)/2}$ in~\eqref{eq:prp.quasi-uni}
    and~\eqref{eq:thm.quasi-uni.a}--\eqref{eq:thm.quasi-uni.b}.

  \item \emph{``$0$-quasi unitary equivalence'' is ``unitary
      equivalence'':} If $\delta=0$ then $J$ is $0$-quasi unitary if
    and only if $J$ is unitary with $J^*=J'$.  Moreover, $\qf d$ and
    $\wt {\qf d}$ are $0$-quasi unitarily equivalent (of order $k \ge
    1$) if and only if $\Delta$ and $\wt \Delta$ are unitarily
    equivalent (in the sense that $JR=\wt RJ$,
    see~\eqref{eq:prp.quasi-uni}).  In this sense, $\delta$-quasi
    unitary equivalence is a \emph{quantitative generalisation of
      unitary equivalence.}

  \item \emph{``$\delta_\eps$-quasi unitary equivalence'' (with
      $\delta_\eps \to 0$) is a generalisation of ``norm resolvent
      convergence'':} As already mentioned in \Prp{quasi-uni},
    $\delta_\eps$-quasi unitary equivalence implies generalised norm
    resolvent convergence.  If, additionally, $k \in \{1,2\}$ and the
    Hilbert spaces are all the same then we also have \emph{classical
      norm resolvent convergence} (see the discussion after
    \Def{gnrs}).
  \end{enumerate}
\end{remark}
%-----------------------------------------------------------------------

We also have the following functional calculus result:
%-----------------------------------------------------------------------
\begin{theorem}[{see~\cite[Sec.~4.2, Thm.~4.2.11, Lem.~4.2.13]{post:12}}]
  \label{thm:quasi-uni}
  \begin{subequations}
    Let $U \subset (-1,\infty)$ be open and unbounded, and let $\map
    \phi {[0,\infty)} \R$ be analytic on $U$ such that $\lim_{\lambda
      \to \infty}\phi(\lambda)$ exists, then there exists a constant
    $C_\phi$ depending only on $\phi$ and $U$ such that
    \begin{equation}
      \label{eq:thm.quasi-uni.a}
      \norm{(J \phi(\Delta) - \phi(\wt \Delta) J) R^{(k-2)/2}} 
      \le C_\phi \delta
    \end{equation}
    for all $\qf d$ and $\wt {\qf d}$ being $\delta$-quasi unitary
    equivalent energy forms (of order $k \ge 1$) with $\spec \Delta
    \subset U$ or $\spec {\wt \Delta} \subset U$.  Moreover, if $k \in
    \{1,2\}$ then we can replace~\eqref{eq:thm.quasi-uni.a} by
    \begin{equation}
      \label{eq:thm.quasi-uni.b}
      \norm{\phi(\wt \Delta) - J\phi(\Delta) J'} 
      \le 5C'_\phi \delta + C_\phi \delta,
      \quadtext{where}
      C'_\phi:=\sup_{\lambda \in U}
    (\lambda+1)^{1/2}\abs{\phi(\lambda)}.
    \end{equation}
  \end{subequations}
\end{theorem}
%-----------------------------------------------------------------------
In particular, if $\phi=\1_{[a,b]}$ with $a,b \notin \spec \Delta$
then~\eqref{eq:thm.quasi-uni.a}--\eqref{eq:thm.quasi-uni.b} are norm
estimates of \emph{spectral projections}.  Moreover, if
$\phi_t(\lambda)=\e^{-t\lambda}$ for $t>0$, then we have norm
estimates of the \emph{heat operators}. One can also prove other
sandwiched versions.  If $\phi$ is only continous on $U$, then one has
to replace $C_\phi \delta$ by $\delta_\phi$ with $\delta_\phi \to 0$
as $\delta \to 0$.

As a conclusion, spectral convergence as in \Thm{spectrum} follows.
Note that we also have convergence of eigenfunctions in energy norm,
namely we can replace~\eqref{eq:conv.ef} by
\begin{equation*}
  \norm[1]{J^1 \psi_0 - \psi_\eps} \le C_1' \delta_\eps\to 0
\end{equation*}
as $\eps \to 0$, see~\cite[Prp.~2.5]{post-simmer:pre17a}.

%-----------------------------------------------------------------------
%
% 
\section{Laplacians on manifolds}
\label{sec:mfd-lapl}
%
%-----------------------------------------------------------------------

%-----------------------------------------------------------------------
\subsection{Energy form, Laplacian and Sobolev spaces associated with
  a Riemannian manifold}
\label{sec:form.laplacian}
%-----------------------------------------------------------------------

Let $(X,g)$ be a complete\footnote{Most of the results are also true
  for incomplete manifolds, but then we have some more technicalities
  which we want to avoid in this presentation.} Riemannian manifold of
dimension $m \ge 2$, for simplicity without boundary.  Denote by $\dd
g$ the Riemannian measure induced by the metric $g$ on $X$ (we often
omit the measure if it is clear from the context).  Then $\Lsqr X =
\Lsqr {X,g}$ is the usual $\Lsqrspace$-space with norm given by
\begin{equation*}
  \normsqr[\Lsqr {X,g}] u := \int_X\abssqr u \dd g.
\end{equation*}
The \emph{energy form} associated with $(X,g)$ is defined by
\begin{equation*}
  \qf d_{(X,g)}(u) := \int_X \abssqr[g]{d u} \dd g
\end{equation*}
for $u$ in the first Sobolev space $\Sob X := \Sob {X,g}$, which can
be defined as the completion of smooth functions with compact
support, under the so-called \emph{energy norm} given by
\begin{equation*}
  \normsqr[\Sob {X,g}] u 
  := \int_X\bigl(\abssqr{u} + \abssqr[g]{d u} \bigr) \dd g.
\end{equation*}
Here, $du$ is a section into the cotangent bundle $T^*M$ and $g$ the
corresponding metric on it.  Note that by definition, $\qf d_{(X,g)}$
is a closed form with $\dom \qf d_{(X,g)}=\Sob{X,g}$.  The
\emph{Laplacian} $\Delta_{(X,g)}$ associated with $(X,g)$ is the
energy operator associated with the energy form $\qf d_{(X,g)}$.  The
Laplacian is a self-adjoint non-negative operator and hence introduces
a scale of Hilbert spaces $\HS^k:= \Sob[k]{\laplacian{(X,g)}}:=\dom
((\laplacian {(X,g)}+1)^{k/2})$ with norm
\begin{equation*}
  \norm[{\Sob[k] {\laplacian{(X,g)}}}] u
  := \norm[\Lsqr{X,g}]{(\laplacian {(X,g)}+1)^{k/2} u}
\end{equation*}
and the extension to negative exponents as already explained in the
text after~\eqref{eq:def.abstr.sob.norm}.  We also call
$\Sob[k]{\laplacian{(X,g)}}$ the \emph{$k$-th Laplacian-Sobolev
  space}.  Obviously, we have $\Sob{X,g}=\Sob{\laplacian{(X,g)}}$ with
identical norms.

If $X$ is a manifold with (smooth) boundary, then we define the
\emph{Neumann energy form} $\qf d^\Neu_{(X,g)}$ as above with domain
$\dom \qf d^\Neu_{(X,g)}=\Sob{X,g}$, where the latter is the closure
of all functions smooth up to the boundary and with compact support
with respect to the energy norm.  The corresponding operator
$\laplacianN{(X,g)}$ is called the \emph{Neumann Laplacian on
  $(X,g)$}.  

Similarly, we define the \emph{Dirichlet energy form} $\qf
d^\Dir_{(X,g)}$ as above with domain $\dom \qf
d^\Dir_{(X,g)}=\Sobn{X,g}$, where the latter is the closure of all
functions with compact support \emph{away from the boundary} with
respect to the energy norm.  The corresponding operator
$\laplacianD{(X,g)}$ is called the \emph{Dirichlet Laplacian on
  $(X,g)$}.

We denote by $\Lsqr{T^*X^{\otimes k},g}$ the $\Lsqrspace$-space of
$k$-tensors with the pointwise norm on the tensors induced by $g$,
i.e., of sections into $T^*X^{\otimes k}=T^*X \otimes \dots \otimes
T^*X$ with norm given by
\begin{equation*}
  \normsqr[\Lsqr {T^*X^{\otimes k},g}] u 
  := \int_X\abssqr[g] u \dd g,
\end{equation*}
where $\abssqr[g]\cdot$ is the canonical extension of $g$ onto the
corresponding tensor bundle.  Here and in the sequel, we are often
sloppy and just write $\normsqr[\Lsqr{X,g}] u$ for the corresponding
norm (assuming that the fibre norm $\abs[g]\cdot$ is clear from the
context).

Denote by $\nabla$ the extension of the Levi-Civita connection on the
tensor bundle $T^*X^{\otimes k}$.  For $k=0$, we have $\nabla u=d u$.
Moreover, we set $\nabla^2 u:=\nabla \nabla u$, which is in $T^*X
\otimes T^*X$ if $u$ is a function.  We set $\nabla^2_{V_1,V_2}:=
\nabla_{V_1} \nabla_{V_2}$ for vector fields $V_1$, $V_2$, and
similarly for higher derivatives.  We say that $u$ has a \emph{$k$-th
  weak derivative} if for all vector fields $V_1,\dots, V_k$, there
exists a measurable function $v\in\Lsymb_{1,loc}(X)$ such that for all
$\phi \in \Cci X$
\begin{equation*}
  \int_X u \cdot \nabla^k_{V_1,\dots,V_k} \phi \dd g 
  = \int_X v \cdot \phi \dd g.
\end{equation*}
We then set $\nabla^k_{V_1,\dots,V_k} u :=v$  and hence defines a section 
$\nabla^k u$ into $T^*X^{\otimes  k}$.  In particular, we set
\begin{gather*}
  \Sobx[k]{p} {X,g} := 
  \bigset{u \in \Lp[p] {X,g}}
  {\text{$u$ has weak derivatives up to order $k$ in $\Lp[p] {X,g}$}},\\
  \intertext{with norm given by}
  \norm[{\Sobx[k]{p} {X,g}}] u ^p
  := \sum_{j=0}^k \norm[{\Lp[p] {T^*X^{\otimes j},g}}] {\nabla^j u}^p
\end{gather*}
for $p\geq 1$, and $\Sob[k]{X,g} :=\Sobx[k]{2} {X,g}$.

Note that the above defined Sobolev space $\Sob {X,g}$ agrees with the
one defined in the beginning of the section, i.e., $\Sob {X,g} = \dom
\qf d_{(X,g)}=\Sob {\laplacian {(X,g)}}$ and the corresponding norms
agree.

%-----------------------------------------------------------------------
\subsection{Bounded geometry, harmonic radius and Euclidean balls}
\label{sec:bdd.geo}
%-----------------------------------------------------------------------

The equivalence of Sobolev spaces and Laplacian-Sobolev spaces is not
given for higher order without further assumptions:
%-----------------------------------------------------------------------
\begin{definition}
  \label{def:bdd.geo}
  We say that a complete Riemannian manifold $(X,g)$ has \emph{bounded
    geometry} if the injectivity radius is uniformly bounded from
  below by some constant $\iota_0>0$ and if the Ricci tensor $\Ric$ is
  uniformly bounded from below by some constant $\kappa_0 \in \R$, i.e.,
  \begin{equation}
    \label{eq:def.kappa0}
    \Ric_x\geq \kappa_0 g_x \qquad\text{for all $x \in X$}
  \end{equation}
  as symmetric $2$-tensors.
\end{definition}
%-----------------------------------------------------------------------
We will not need assumptions on \emph{derivatives} of the curvature
tensor (i.e., bounded geometry of higher order) in this article.

%-----------------------------------------------------------------------
\begin{proposition}[{\cite[Prp.~2.10]{hebey}}]
  \label{prp:ell.reg}
  Suppose that $(X,g)$ is a complete manifold with bounded geometry,
  then the set of smooth functions with compact support $\mcD(X)$ is
  dense in the Sobolev space $\Sob[2]{X,g}$ and the norms of
  $\Sob[2]{X,g}$ and $\Sob[k]{\laplacian{(X,g)}}$ are equivalent,
  i.e., there are constants $\Cellreg \geq \cellreg>0$ such that
  \begin{equation*}
    \cellreg \norm[\Lsqr{X,g}] {(\laplacian{(X,g)}+1)f}
    \le \norm[{\Sob[2]{X,g}}] f
    \le \Cellreg \norm[\Lsqr{X,g}] {(\laplacian{(X,g)}+1)f}
  \end{equation*}
  for all $f \in \Sob[2]{X,g}$, where $\Cellreg$ depends only on a lower
  bound $\kappa_0$ of the Ricci curvature.
\end{proposition}
%-----------------------------------------------------------------------
The last estimate is a conclusion from the following result:
\begin{proposition}[\cite{aubin:76}]
  \label{prp:ell.reg2}
  Assume that $(X,g)$ is complete. Then we have
  \begin{equation*}
    \normsqr[\Lsqr{T^*X^{\otimes 2}}] {\nabla^2 u}
    = \normsqr[\Lsqr{X,g}] {\laplacian{(X,g)} u}
      - \iprod[\Lsqr{T^*X,g}]{\Ric du} {du}
    % \int_X \abssqr[g] {\nabla^2 u} \dd g
    % = \int_X 
    %      \bigl(
    %         \abssqr{\laplacian{(X,g)} u}
    %         - \Ric (du, du)
    %      \bigr)
    %   \ds g
  \end{equation*}
  for all $u \in \mcD(X)$
  where we understand $\Ric$ as endomorphism on $T^*X$.
\end{proposition}
%-----------------------------------------------------------------------
\begin{proof}
  We apply the Bochner-Lichnerowicz-Weitzenb\"ock formula on
  $1$-forms: $\nabla^* \nabla \omega=\laplacian{(T^*X,g)} \omega -
  \Ric(\omega,\omega)$ to the $1$-form $\omega=du=\nabla u$ for
  $u\in\mcD (X)$.  As $\laplacian{(T^*X,g)} du= d\laplacian{(X,g)} u$,
  we conclude
  \begin{align*}
    \iprod {\nabla^2 u} {\nabla^2 u}
    = \iprod {\nabla^*\nabla^2 u} {\nabla u}
    &= \iprod {\laplacian{(T^*X,g)} du} {d u}
      - \iprod{\Ric du}{du}\\
    &= \iprod {d \laplacian{(X,g)} u} {d u}
      - \iprod{\Ric du}{du}\\
    &= \iprod {\laplacian{(X,g)} u} {d^*d u}
      - \iprod{\Ric du}{du}
  \end{align*}
  for sufficiently smooth functions $u$ such that all integrals exist.
  Here, we used $\laplacian{(T^*X,g)} d=d d^* d= d \laplacian{(X,g)}$
  and $\laplacian{(X,g)}=d^*d$.
\end{proof}
%-----------------------------------------------------------------------
If the Ricci curvature is bounded from below, we conclude the equality
of the spaces $\Sob[2]{X,g}$ and $\Sob[2]{\laplacian{(X,g)}}$.

%-----------------------------------------------------------------------
\begin{proposition}[{\cite[Th.~1.3]{hebey}}]
  \label{prp:eucl.metric}
  Assume that $(X,g)$ is complete and has bounded geometry (with
  constants $\kappa_0\in\R$ and $\iota_0>0$).  Then for all $a \in
  (0,1)$ there exist $r_0>0$, $K\ge 1$ and $k>0$ depending only on
  $\kappa_0$, $\iota_0$ and $a$, such that around any point $x\in X$
  there exist harmonic charts $\phi_x=(y^1,\dots,y^m)$ defined on
  $\clo{B_{r_0}(x)}$, and in these charts we have
  \begin{subequations}
    \begin{gather}
      \label{eq:eucl.metric.a}
      K^{-1}\delta_{ij}\leq g_{ij}\leq K\;\delta_{ij}\\
      \label{eq:eucl.metric.b}
      \abs{g_{ij}(x')-g_{ij}(x'')}\leq k\; d_g(x',x'')^a.
    \end{gather}
  \end{subequations}
  for all $x'$, $x'' \in B_{r_0}(x)$.
\end{proposition}
%-----------------------------------------------------------------------
The radius $r_0$ will be called \emph{harmonic radius} in the
following.  We refer to~\cite{hpw:14,hebey,hebey2} and the references
therein for more details.  We assume $r_0 \le 1$ here, as it
simplifies some estimates later on, when using estimates of cut-off
functions on small balls, see e.g.\ \Lem{0}.
% so the manifold has bounded geometry in the sense of Grigor'yan 
% \cite{grigoryan} p.312. 

Denote by $g_{\eucl,x}$ the euclidean metric in the harmonic chart
$\phi_x$ defined in the ball $B_r(x)$ by
\begin{equation}
  \label{eq:eucl.met}
  g_\eucl(\partial_{y_i},\partial_{y_j})=\delta_{ij}.
\end{equation}
We immediately conclude from \Prp{eucl.metric}:
%-----------------------------------------------------------------------
\begin{corollary}
  \label{cor:eucl.metric}
  Let $p \in X$ and $B:=B_r(p)$ then
  \begin{enumerate}
  \item
    \label{eucl.metric.a}
     the volume measures and the cotangent norm satisfy the
    estimates
    \begin{equation}
      \label{eq:density.eucl}
      K^{-m/2} \dd g_\eucl 
      \le \dd g_x 
      \le K^{m/2} \dd g_\eucl
      \quadtext{and}
       K^{-1}\abssqr[g_\eucl]{\xi}
      \le \abssqr[g_x]{\xi}
      \le K^1  \abssqr[g_\eucl]{\xi}
    \end{equation}
    for all $x \in B$ and $\xi \in T_x^*X$;
  \item 
    \label{eucl.metric.b}
    we have the following norm estimates
    \begin{align*}
      K^{-m/4} \norm[\Lsqr{B,g_\eucl}] u
      &\le \norm[\Lsqr{B,g}] u
      \le K^{m/4} \norm[\Lsqr{B,g_\eucl}] u,\\
      K^{-(m+2)/4} \norm[\Lsqr{T^*B,g_\eucl}] {du}
      &\le \norm[\Lsqr{T^*B,g}] {du}
      \le K^{(m+2)/4} \norm[\Lsqr{T^*B,g_\eucl}] {du},\\
      K^{-(m+2)/4} \norm[\Sob{B,g_\eucl}] u
      &\le \norm[\Sob{B,g}] u
      \le  K^{(m+2)/4} \norm[\Sob{B,g_\eucl}] u
    \end{align*}
    for all $u \in \Lsqr{B,g}$ resp.\ $u \in \Sob{B,g}$.
  \end{enumerate}
\end{corollary}
%-----------------------------------------------------------------------

%-----------------------------------------------------------------------
\subsection{The non-concentrating property}
\label{sec:non-concentr}
%-----------------------------------------------------------------------

We now formulate a property which will be used in all our examples.
Typically, $A=A_\eps$ and $\delta_\eps \to 0$ as $\eps \to 0$; the
name ``non-concentrating'' comes from the fact that if $f=f_\eps$ is
an eigenfunction with eigenvalue $\lambda_\eps$ bounded in $\eps$,
then $f_\eps$ cannot concentrate on $A_\eps$ as $\eps \to 0$.
%-----------------------------------------------------------------------
\begin{definition}
  \label{def:non-concentr}
  Let $(X,g)$ be a Riemannian manifold, $A \subset B \subset X$ and
  $\delta>0$.  We say that $(A,B)$ is
  $\delta$-\emph{non-concentrating} (of order $1$) if
  \begin{equation}
    \label{eq:non-concentr}
    \norm[\Lsqr{A,g}] f
    \le \delta \norm[\Sob{B,g}] f
  \end{equation}
  for all $f \in \Sob{B,g}$.
\end{definition}
%-----------------------------------------------------------------------
Note that if $\wt B \supset B$ and if $(A,B)$ is
$\delta$-non-concentrating, then also $(A,\wt B)$ is
$\delta$-non-concentrating.

Once we have the non-concentrating property, we can immediately
conclude a similar estimate for the derivatives:
%-----------------------------------------------------------------------
\begin{proposition}
  \label{prp:non-concentr2}
  Assume that $(A,B)$ is $\delta$-non-concentrating, then $(A,B)$ is
  $\delta$-non-concentrating of order $2$, i.e.,
  \begin{equation}
    \label{eq:non-concentr2}
    \norm[\Lsqr{A,g}] {df}
    \le \delta \norm[{\Sob[2]{B,g}}] f
  \end{equation}
  for all $f \in \Sob[2]{B,g}$.
\end{proposition}
%-----------------------------------------------------------------------
\begin{proof}
  Let $f \in \Sob[2]{X,g}$.  We apply~\eqref{eq:non-concentr} to the
  function $\phi=\abs[g]{df}$ and calculate for any $x \in X$ with $df(x) \neq
  0$ and any $V \in T_x X$:
  \begin{equation}
    \label{eq:non-concentr3}
    d_V \phi
    =d_V \sqrt{\iprod[g] {df}{df}}
    =\frac 1 {\sqrt{\iprod[g] {df}{df}}}
    \iprod[g] {\nabla_V df}{df}.
 \end{equation}
 We conclude $\abs{d_V \phi} \le \abs[g]{\nabla^2 f}\abs[g] V$
 by the Cauchy-Schwarz inequality.  In particular, $\abs[g]{d \phi}
 \le \abs[g]{\nabla^2 f}$, and this inequality is also true if
 $df(x)=0$.  Inequality~\eqref{eq:non-concentr} now yields
 \begin{align*}
   \norm[\Lsqr{A,g}] {df}
   = \norm[\Lsqr{A,g}] \phi
   &\le \delta \norm[\Sob{B,g}] \phi
   = \delta \bigl(\normsqr[\Lsqr{B,g}] {df} 
       + \normsqr[\Lsqr{B,g}] {d \phi}\bigr)^{1/2}\\
   & \le \delta \bigl(\normsqr[\Lsqr{B,g}] {df} 
       + \normsqr[\Lsqr{B,g}]{\nabla^2 f}\bigr)^{1/2}
   \le \delta \norm[{\Sob[2]{B,g}}] f.  \qedhere
 \end{align*}
\end{proof}
%-----------------------------------------------------------------------

Let us now check the non-concentrating property for balls of different
radii.  Here,
\begin{equation}
  \label{eq:def.ball}
  B_r(p)=\set{x \in X}{d_g(x,p)<r}.
\end{equation}
% -----------------------------------------------------------------------
\begin{lemma}
  \label{lem:0}
  Assume that $(X,g)$ has bounded geometry with harmonic radius $r_0
  \in (0,1]$.  Let $\eta \in (0,r_0)$ and $\eps \in (0,\eta/2)$ then
  $(B_\eps(p),B_\eta(p))$ are $\tau_m(\eps/\eta)$-non-concentrating
  for all $p \in X$, i.e.,
  \begin{equation*}
    \norm[\Lsqr{B_\eps(p),g}] f
    \le \tau_m \Bigl(\frac \eps {\eta} \Bigr) 
    \norm[\Sob{B_\eta(p),g}] f
  \end{equation*}
  for all $f \in \Sob{B_\eta(p),g}$.  Here,
  \begin{equation}
    \label{eq:tau.m}
    \tau_m(\omega)
    := \sqrt 8 K^{(m+1)/2}\omega
    \quadtext{resp.}
     \tau_2(\omega)
    := \sqrt 8 K^{3/2}\omega \sqrt{\abs{\log \omega}}
  \end{equation}
  if $m \ge 3$ resp.\ $m=2$.
\end{lemma}
%-----------------------------------------------------------------------
\begin{proof}
  We apply the results of~\cite[Sec.~A.2]{post:12}.  We first consider
  Euclidean balls: note that in polar coordinates the Euclidean metric
  is a warped product $g_\eucl=\dd s^2 + s^2 h$ with density function
  $\rho(s)=s^{m-1}$, where $h$ is the standard metric on the
  $(m-1)$-dimensional sphere.  We then
  apply~\cite[Cor.~A.2.7~(A.9b)]{post:12} with $s_0=0$, $s_1=\eps$,
  $s_2=\eta$, $a=\eta-\eps$.  We conclude
  \begin{equation*}
    \normsqr[\Lsqr{B_\eps,g_\eucl}]f
    \le 2 \eta_2(0,\eps,\eta)
    \Bigl(\normsqr[\Lsqr{B_\eta,g_\eucl}] {f'}
      + \frac 1{(\eta-\eps)^2}\normsqr[\Lsqr{B_\eta,g_\eucl}] f
    \Bigr),
  \end{equation*}
  where $f'$ denotes the radial derivative and where
  \begin{equation*}
    \eta_2(0,\eps,\eta)
    := \int_0^\eps
       \Bigl(\int_t^\eta \frac 1 {\rho(s)} \dd s \Bigr) \rho(t) \dd t
    \le
    \begin{cases}
      \eps^2 \log(\eta/\eps),& \text{if $m=2$},\\
      \eps^2,& \text{if $m\ge2$},
    \end{cases}
  \end{equation*}
  provided $\eps \le \eta/2 < \e^{-1/2}\eta$.  In particular,
  \begin{equation*}
    \frac {\eps^2}{(\eta-\eps)^2}
    =\frac{\omega^2}{(1-\omega)^2}
    \le 4\omega^2 
  \end{equation*}
  with $\omega=\eps/\eta \le 1/2$.  We then use
  \Corenum{eucl.metric}{eucl.metric.b} to carry over the estimates to
  the original metric $g$, namely
  \begin{align*}
    \normsqr[\Lsqr{B_\eps(p),g}] f
    \le K^{m/2} \normsqr[\Lsqr{B_\eps,g_\eucl}] f
    &\le 8K^{m/2} [\abs{\log \omega}] \omega^2\normsqr[\Sob{B_\eta,g_\eucl}] f\\
    &\le 8K^{m+1} [\abs{\log \omega}] \omega^2\normsqr[\Sob{B_\eta,g}] f,
  \end{align*}
  where $[\abs{\log \omega}]$ appears only if $m=2$.
\end{proof}
%-----------------------------------------------------------------------

%-----------------------------------------------------------------------
\subsection{The non-concentrating property for many balls}
\label{sec:non-concentr-balls}
%-----------------------------------------------------------------------

%-----------------------------------------------------------------------
\begin{definition}
  \label{def:separated}
  We denote by
  \begin{equation}
    \label{eq:eps.nbhd}
    B_r(I) :=\bigset{x \in X}{d_g(x,I):=\inf_{p \in I} d_g(x,p) \le r}
  \end{equation}
  the \emph{$r$-neighbourhood} of a subset $I \subset X$.  We say that
  $I \subset X$ is an \emph{$r$-separated set} if for all $p_1,p_2 \in
  I$, $p_1 \ne p_2$, we have $d(p_1,p_2) \ge 2r$.
\end{definition}
%-----------------------------------------------------------------------
Let $I$ be an $\eta$-separated set in $X$, then $B_\eps(I)$ consists
of $\card I$-many \emph{disjoint} balls of radius $\eps \in (0,\eta)$
around each point in $I$.

Let us now check the non-concentrating property for the union of balls:
%-----------------------------------------------------------------------
\begin{proposition}
  \label{prp:0}
  Let $(X,g)$ be a complete Riemannian manifold with bounded geometry
  and harmonic radius $r_0>0$.  Let $\eta \in (0,r_0)$ and $\eps \in
  (0,\eta/2)$.  Assume that $I$ is $\eta$-separated, then
  $(B_\eps(I), B_\eta(I))$ are
  $\tau_m(\eps/\eta)$-separated, i.e.,
  \begin{equation*}
    \norm[\Lsqr{B_\eps(I),g}] f
    \le \tau_m\bigg(\frac \eps \eta \bigg) \norm[\Sob{B_\eta(I),g}] f
  \end{equation*}
  for all $f \in \Sob {B_\eta(I),g}$.
\end{proposition}
%-----------------------------------------------------------------------
\begin{proof}
  The estimate follows from
  \begin{equation*}
    \normsqr[\Lsqr{B_\eps(I),g}] f
    = \sum_{p \in I} \normsqr[\Lsqr{B_\eps(p),g}] f
    \le \sum_{p \in I} 
       \tau_m\bigg(\frac \eps \eta \bigg) 
       \normsqr[\Sob{B_\eta(p),g}] f
    = \tau_m\bigg(\frac \eps \eta \bigg) 
        \normsqr[\Sob{B_\eta(I),g}] f
  \end{equation*}
  using \Lem{0} and the disjointness of the balls in $B_\eta$.
\end{proof}
%-----------------------------------------------------------------------

%-----------------------------------------------------------------------
%
\section{Neumann obstacles without an effect}
\label{sec:neu-fading}
%
%-----------------------------------------------------------------------

%-----------------------------------------------------------------------
\subsection{Abstract Neumann obstacles without effect}
\label{sec:neu-fading-abstr}
%-----------------------------------------------------------------------

Let $(X,g)$ be a Riemannian manifold of dimension $m \geq 2$ and let
$B_\eps \subset X$ be a closed subset for each $\eps \in (0,\eps_0]$.
We will impose conditions on the family $(B_\eps)_\eps$ such that the
Neumann Laplacian on $X_\eps := X \setminus B_\eps$ converges to the
Laplacian on $X$.  Later, we will specify some examples for $B_\eps$
and show that one can actually realise such obstacles having the
properties as e.g.\ in the following definition:

%-----------------------------------------------------------------------
\begin{definition}
  \label{def:neu-fading}
  We say that a family $(B_\eps)_\eps$ of closed subsets of a
  Riemannian manifold $(X,g)$ is \emph{Neumann-asymptotically fading}
  if the following conditions are fulfilled:
  \begin{enumerate}
  \item
    \label{neu-fading.a}
    \emph{Non-concentrating property:} We assume that $(B_\eps,X)$ is
    $\delta'_\eps$-non-concentrating with $\delta'_\eps \to 0$.
  \item
    \label{neu-fading.a'}
    \emph{Elliptic regularity:} We assume that $(X,g)$ is elliptically
    regular, i.e., that there is $\Cellreg \ge 1$ such that
    \begin{equation*}
      \norm[{\Sob[2]{X,g}}] f
      \le \Cellreg \norm[\Lsqr{X,g}]{(\laplacian{(X,g)}+1)f}
    \end{equation*}
    for all $f \in \Sob[2]{\laplacian{(X,g)}}=\dom \laplacian{(X,g)}$.
  \item
    \label{neu-fading.b}
    \emph{Uniform extension property:} We assume that there is a
    constant $\Cext \ge 1$ such that $\norm{E_\eps} \le \Cext$ for all
    $\eps \in (0,\eps_0]$, where
    \begin{equation*}
      \map{E_\eps}{\Sob{X_\eps,g}}{\Sob{X,g}}
    \end{equation*}
    is an extension operator, i.e., $(E_\eps u)\restr {X_\eps}=u$ for
    all $u \in \Sob{X_\eps,g}$.
  \end{enumerate}
\end{definition}
%-----------------------------------------------------------------------

We now show our first main result:
%-----------------------------------------------------------------------
\begin{theorem}
  \label{thm:neu-fading}
  Let $(X,g)$ be a Riemannian manifold and $(B_\eps)_\eps$ be a family
  of closed subsets of $X$.  If $(B_\eps)_\eps$ is
  Neumann-asymptocially fading, then the energy form $\qf d_{(X,g)}$
  of $(X,g)$ and the (Neumann) energy form $\qf d^\Neu_{(X_\eps,g)}$
  of $(X_\eps,g)$ with $X_\eps=X\setminus B_\eps$ are
  $\delta_\eps$-quasi-unitarily equivalent of order $k=2$ with
  $\delta_\eps=\Cext \Cellreg \delta'_\eps$.
\end{theorem}
%-----------------------------------------------------------------------
\begin{proof}
  We show that the hypotheses of \Def{quasi-uni} are fulfilled.  To do
  so, we first need to specify the spaces and transplantation
  operators.  Namely, we set
  \begin{align*}
    \map J {\HS:=\Lsqr{X,g}} {&\wt \HS:=\Lsqr{X_\eps,g}},
    & f & \mapsto f \restr {X_\eps},\\
    \map {J^1} {\HS^1:=\Sob{X,g}} {&\wt \HS^1:=\Sob{X_\eps,g}},
    & f & \mapsto f \restr {X_\eps}\\
   \map{J'}{\wt \HS=\Lsqr{X_\eps,g}}{&\HS=\Lsqr{X,g}},   
     & u & \mapsto \bar u,\\
   \map{J^{1\prime}}{\wt\HS^1=\Sob{X_\eps,g}}{&\HS^1=\Sob{X,g}}, 
     &  u & \mapsto E_\eps u,
   \end{align*}
   where $\bar u$ denotes the extension of $\map u {X_\eps} \C$ by $0$
   on $B_\eps$.

   We check the hypotheses of \Def{quasi-uni}: We easily see that
   \begin{equation*}
     J'=J^*, \qquad
     JJ'=\id_{\wt \HS} \quadtext{and}
     J^1=J \restr{\HS^1}.
   \end{equation*}
   Moreover, we have
   \begin{equation*}
    \normsqr[\Lsqr{X_\eps,g}]{Jf} 
    = \int_{X_\eps} \abssqr f \dd g 
    \le \int_X \abssqr f \dd g 
    = \normsqr[\Lsqr{X,g}] f,
  \end{equation*}
  and if $\supp f \subset X_\eps$, then $\norm{Jf}=\norm f$, hence we
  have $\norm J =1$; in particular,~\eqref{eq:quasi-uni.a} is
  fulfilled with $\delta=0$.

  The first estimate in~\eqref{eq:quasi-uni.b} follows that
  $(B_\eps,X)$ is $\delta'_\eps$-non-concentrating
  (see~\eqref{eq:non-concentr}), namely we have
  \begin{equation*}
    \norm[\Lsqr{X,g}] {f-J'Jf}
    = \norm[\Lsqr{B_\eps,g}] f
    \le \delta'_\eps \norm[\Sob{X,g}] f.
  \end{equation*}
  Moreover, $J^{1\prime}u-J'u=\1_{B_\eps} E_\eps u$ (the uniform
  extension onto $B_\eps$), hence
  \begin{equation*}
    \norm[\Lsqr{X,g}]{J^{1\prime}u-J'u}
    = \norm[\Lsqr{B_\eps,g}]{E_\eps u}
    \le \delta'_\eps \norm[\Sob{X,g}]{E_\eps u}
    \le \delta'_\eps \Cext \norm[\Sob{X_\eps,g}] u
  \end{equation*}
  by the non-concentrating property~\eqref{eq:non-concentr} and the
  uniform extension property \Defenum{neu-fading}{neu-fading.b}.
  Finally,
  \begin{align*}
    \bigabs{\qf d_\eps(J^1f,u) - \qf d(f,J^{1\prime} u)}
    &= \bigabs{\iprod[\Lsqr {B_\eps,g}] {df} {d (E_\eps u)}}\\
    &\le \norm[\Lsqr {B_\eps,g}] {df} 
         \norm[\Lsqr {B_\eps,g}] {d (E_\eps u)}\\
    &\le \delta'_\eps 
       \norm[{\Sob[2]{X,g}}] f
       \Cext \norm[\Sob {X_\eps,g}] u\\
    &\le \Cext \Cellreg \delta'_\eps 
      \norm[\Lsqr{X,g}]{(\laplacian{(X,g)}+1)f} 
      \norm[\Sob{X_\eps,g}] u
  \end{align*}
  by the non-concentrating property~\eqref{eq:non-concentr2}, the
  elliptic regularity assumption \Defenum{neu-fading}{neu-fading.a'}.
  and again the uniform extension property
  \Defenum{neu-fading}{neu-fading.b}.
\end{proof}

%-----------------------------------------------------------------------
\subsection{Application: many small balls as Neumann obstacles}
\label{ssec:neu-fading}
%-----------------------------------------------------------------------

We now let $B_\eps$ be the disjoint union of many balls: Assume that
for each $\eps>0$ there is $\eta_\eps$ such that $\eps/\eta_\eps \to
0$ (e.g., $\eta_\eps=\eps^\alpha$ for some $0<\alpha<1$).  Assume
additionally, that $(I_\eps)_\eps$ is a family of
$\eta_\eps$-separated subsets $I_\eps\subset X$ (i.e., different
points in $I_\eps$ have distance at least $2\eta_\eps$, see
\Def{separated}).  We denote by
\begin{equation*}
  B_\eps := B_\eps(I_\eps)
  \qquadtext{and}
  X_\eps=X \setminus B_\eps
\end{equation*}
the $\eps$-neighbourhood of all points in $I_\eps$ resp.\ its
complement in $X$.  Note that --- by the $\eta_\eps$-separation ---
$B_\eps$ consists of $\card{I_\eps}$-many \emph{disjoint} balls around
each point in $I_\eps$.

Let us first show the uniform extension property of
\Defenum{neu-fading}{neu-fading.b}: We define
\begin{equation*}
  \map{E_\eps}{\Sob{X_\eps,g}}{\Sob{X,g}}, \qquad
  u \mapsto \wt u,
\end{equation*}
where $\wt u$ denotes the \emph{harmonic extension} on $B_\eps$ with
respect to the \emph{Euclidean} metric $g_\eucl$ on $B_\eps$ (the
metric $g_\eucl$ is defined in~\eqref{eq:eucl.met} on each small
ball).

% %-----------------------------------------------------------------------
% \subsection{Harmonic extensions}
% \label{sec:harm.ext}
% %-----------------------------------------------------------------------

We first need an estimate of the harmonic extension from an annulus:
%-----------------------------------------------------------------------
\begin{lemma}
  \label{lem:harmonic}
  For $0< \eps \le 1$, let $B_\eps$ and $B_{2\eps}$ be Euclidean balls
  in $\R^m$ of radius $\eps$ and $2\eps$ around $0$.  For $u \in
  \Sob{B_{2\eps} \setminus B_\eps}$, denote by $\wt u$ the
  \emph{harmonic extension} of $u$ into $B_\eps$.  Then $\wt u \in
  \Sob {B_\eps}$ and there exist constants $C_0,C_1>0$ depending only
  on $m$ such that
    \begin{equation*}
      \int_{B_\eps} \abssqr {\wt u} 
      \leq C_0\int_{B_{2\eps} \setminus B_\eps} (\abssqr u +\eps^2 \abssqr{du})
      \quadtext{and}
      \int_{B_\eps} \abssqr {\dd \wt u} 
      \leq C_1\int_{B_{2\eps} \setminus B_\eps} \abssqr{du}
      % \normsqr[\Lsqr{B_{2\eps} \setminus B_\eps,g}] {d\wt u}
      % \le C_1 \normsqr[\Sob{B_{2\eps} \setminus B_\eps}] u
    \end{equation*}
    for all $u \in \Sob{B_{2\eps} \setminus B_\eps}$.
\end{lemma}
%-----------------------------------------------------------------------
\begin{proof}
  This result is given in~\cite{rauch-taylor:75}. For the convenience
  of the reader, we repeat the proof using a scaling argument here:

  For $u \in \Sob{B_{2\eps} \setminus B_\eps}$ let $f(x)=u(\eps x)$.
  Then $f\in \Sob{B_2 \setminus B_1}$ and we have the scaling
  behaviour
  \begin{equation*}
    \int_{B_2 \setminus B_1} \abssqr f 
    =\eps^{-m} \int_{B_{2\eps} \setminus B_\eps} \abssqr u
    \quadtext{and}
    \int_{B_2 \setminus B_1} \abssqr {df} 
    =\eps^{2-m} \int_{B_{2\eps} \setminus B_\eps} \abssqr {du}
  \end{equation*}
  We know that $\map{\wt\cdot}{\Sob{B_2 \setminus B_1}}{\Sob{B_1}}$,
  $f \mapsto \wt f$, is a continuous operator.  In particular, there
  exists a constant $C_0>0$ depending only on $m$ such that
  \begin{equation*}
    \int_{B_1} \bigl(\abssqr {\wt f} +\abssqr{d \wt f}\bigr)
    \leq C_0 \int_{B_2 \setminus B_1} \bigl(\abssqr f + \abssqr{df}\bigr)
  \end{equation*}
  holds.  After scaling, we obtain
  \begin{equation*}
    \int_{B_\eps} \abssqr{\tilde u}
    \leq C_0\int_{B_{2\eps} \setminus B_\eps} \bigl(|u|^2+\eps^2 |du|^2\bigr)
    \leq C_0\int_{B_{2\eps} \setminus B_\eps} \bigl(|u|^2+ |du|^2\bigr)
  \end{equation*}
  as $\eps \le 1$.  For the control of the derivative, we remark that
  the harmonic extension of the constant function $\1$ on $B_2
  \setminus B_1$ is the constant function $\1$ on $B_1$.  Therefore,
  we can assume that $u$ (and after rescaling also $f$) is orthogonal
  to $\1$.  If $\lambda_1$ denote the first positive eigenvalue of the
  Neumann problem of the standard annulus $B_2 \setminus \clo B_1$,
  we can conclude with the min-max principle and obtain
  \begin{equation*}
    %  \iprod[\Lsqr{B_2 \setminus B_1}] f \1 =0 \Rightarrow
    \int_{B_2 \setminus B_1} \abssqr f
    \le \frac 1 {\lambda_1} \int_{B_2 \setminus B_1} \abssqr{df},
    \quadtext{so that}
    \int_{B_1} \abssqr{d\wt f}
    \le C_0 \Bigl(1+\frac 1{\lambda_1} \Bigr) 
          \int_{B_2 \setminus B_1} \abssqr{df}.
  \end{equation*}
  Since both sides scale with the same order, rescaling gives
  \begin{equation*}
    \int_{B_\eps} \abssqr{d\wt u}
    \le \underbrace{C_0 \Bigl(1+\frac 1{\lambda_1} \Bigr)}_{=:C_1} 
          \int_{B_{2\eps} \setminus B_\eps} \abssqr{du}. \qedhere
  \end{equation*}
\end{proof}
%-----------------------------------------------------------------------

%-----------------------------------------------------------------------
\begin{proposition}
  \label{prp:harm.ext}
  Assume $I_\eps$ is $2\eps$-separated for each $\eps \in (0,r_0/2)$.
  Then there is a constant $\Cext>0$ such that
  \begin{equation*}
    \norm[\Sob{B_{2\eps},g}] {\wt u}
   \le \Cext \norm[\Sob{B_{2\eps} \setminus B_\eps,g}] u
  \end{equation*}
  for all $u \in \Sob {X_\eps,g}$ and all $\eps$.  In particular,
  there exists $\Cext \ge 1$ such that $\norm{E_\eps} \le \Cext$ for all $\eps
  \in (0,r_0/2)$.
\end{proposition}
%-----------------------------------------------------------------------
\begin{proof}
  We have
  \begin{align*}
    \normsqr[\Sob{B_\eps,g}] {\wt u}
    =\sum_{p \in I_\eps} \normsqr[\Sob{B_\eps(p),g}] {\wt u}
    &\le K^{m/2+1} \sum_{p \in I_\eps} 
      \normsqr[\Sob{B_\eps(p),g_\eucl}] {\wt u}\\
    &\le K^{m/2+1}  (C_0+C_1) \sum_{p \in I_\eps} 
       \normsqr[\Sob{B_{2\eps}(p) \setminus B_\eps(p),g_\eucl}] u\\
    &\le K^{(m+2)}(C_0+C_1) \sum_{p \in I_\eps} 
       \normsqr[\Sob{B_{2\eps}(p) \setminus B_\eps(p),g}] u\\
    &=: \Cext^2 \normsqr[\Sob{B_{2\eps}\setminus B_\eps,g}] u
  \end{align*}
  using \Corenum{eucl.metric}{eucl.metric.b} and \Lem{harmonic}.
\end{proof}
%-----------------------------------------------------------------------

The proof of the following theorem follows directly from
\Thm{neu-fading} together with \Prp{0} ($(B_\eps,B_\eta(I_\eps))$ and
hence $(B_\eps,X)$) are $\tau_m(\eps/\eta_\eps)$-non-concentrating,
see \Defenum{neu-fading}{neu-fading.a}), \Prp{ell.reg} (for the
elliptic regularity assumption in \Defenum{neu-fading}{neu-fading.a'})
and \Prp{harm.ext}:
%-----------------------------------------------------------------------
\begin{theorem}
  \label{thm:neu-fading.balls}
  Let $(X,g)$ be a complete Riemannian manifold with bounded geometry,
  and let $B_\eps=\bigdcup_{p \in I_\eps} B_\eps(p)$ be the union of
  $\eta_\eps$-separated balls of radius $\eps$.  If $\eps/\eta_\eps
  \to 0$, then $(B_\eps)_\eps$ is Neumann-asymptotically fading, i.e.,
  the energy form $\qf d_{(X,g)}$ and the (Neumann) energy form $\qf
  d^\Neu_{(X_\eps,g)}$ are $\delta_\eps$-quasi-unitarily equivalent of
  order $k=2$ with
  \begin{align*}
    \delta_\eps&=\Err(\eps/\eta_\eps)&& \text{if $m\ge 3$ \qquad resp.}&
    \delta_\eps&=\Err(\sqrt{\log(\eta_\eps/\eps)}\eps/\eta_\eps)
    &&\text{if $m=2$.}
  \end{align*}
  The error depends only on $m$, $K$ and $\kappa_0$,
  see~\eqref{eq:eucl.metric.a} and~\eqref{eq:def.kappa0}.  In
  particular, if $\eta_\eps=\eps^\alpha$ with $\alpha \in (0,1)$, then
  $\delta_\eps=\Err(\eps^{1-\alpha})$ if $m \ge 3$ resp.\
  $\delta_\eps=\Err(\eps^{1-\alpha} \sqrt{\abs{\log \eps}})$ if $m=2$.
\end{theorem}
%-----------------------------------------------------------------------

%-----------------------------------------------------------------------
\begin{remark}
  \label{rem:neu-homogen}
  If $\alpha=1$ i.e., $\eta_\eps/\eps$ converges to a constant, then
  we do not expect the result to be true in general.  If the balls are
  placed on a lattice of order $\eps$, and if their radius is $\eps$,
  then we are in the setting of \emph{homogenisation} (with Neumann
  boundary conditions), and we expect that the limit operator is no
  longer the free Laplacian.
\end{remark}
%-----------------------------------------------------------------------

%-----------------------------------------------------------------------
%
\section{Dirichlet obstacles without an effect}
\label{sec:dir-fading}
%
%-----------------------------------------------------------------------

%-----------------------------------------------------------------------
\subsection{Abstract Dirichlet obstacles without effect}
\label{sec:dir-fading-abstr}
%-----------------------------------------------------------------------

Let us now consider the same problem, but with Dirichlet boundary
conditions on the obstacles:

%-----------------------------------------------------------------------
\begin{definition}
  \label{def:dir-fading}
  We say that a family $(B_\eps)_\eps$ of closed subsets of a
  Riemannian manifold $(X,g)$ is \emph{Dirichlet-asymptotically
    fading} (of order $k \ge 0$) if there exists a sequence
  $(\chi_\eps)_\eps$ of Lipschitz-continuous cut-off functions
  $\map{\chi_\eps}X{[0,1]}$ with $\supp \chi_\eps \subset X_\eps$ such
  that the following conditions are fulfilled:
  \begin{enumerate}
  % \item
  %   \label{dir-fading.a}
  %   \emph{Obstacle cut-off:} The support of the cut-off function is
  %   away from the obstacle, i.e., $\chi_\eps \restr{B_\eps}=0$.
  \item
    \label{dir-fading.b}
    \emph{Non-concentrating property:} We assume that $(B^+_\eps,X)$
    is $\delta'_\eps$-non-concentrating with $\delta'_\eps \to 0$,
    where $B_\eps^+:=\supp(1-\chi_\eps)$.
  \item
    \label{dir-fading.b'}
    \emph{Elliptic regularity:} We assume that $(X,g)$ is elliptically
    regular, i.e., that there is $\Cellreg \ge 1$ such that
    \begin{equation*}
      \norm[{\Sob[2]{X,g}}] f
      \le \Cellreg \norm[\Lsqr{X,g}]{(\laplacian{(X,g)}+1)f}
    \end{equation*}
    for all $f \in \Sob[2]{\laplacian{(X,g)}}=\dom \laplacian{(X,g)}$.

  \item
    \label{dir-fading.c} 
   The cut-off has \emph{moderate decay} of order $k \ge 2$ if
    \begin{equation*}
      \map{T_\eps^+}{\Sob[k]{\laplacian{(X,g)}}}{\Lsqr{T^*B_\eps^+,g}},
      \qquad
      f \mapsto f \restr {B_\eps^+} d \chi_\eps
    \end{equation*}
    has norm $\norm{T_\eps^+}=\delta_\eps^+ \to 0$ as $\eps \to 0$. 
  \end{enumerate}
\end{definition}
%-----------------------------------------------------------------------
Note that we have $B_\eps \subset B_\eps^+$.  If $B_\eps$ is a union
of small balls, then this problem is the famous \emph{crushed ice
  problem} of~\cite{rauch-taylor:75}.

Our next main result is the following:
%-----------------------------------------------------------------------
\begin{theorem}
  \label{thm:dir-fading}
  Let $(X,g)$ be a Riemannian manifold and $(B_\eps)_\eps$ be a family
  of closed subsets of $X$.  If $(B_\eps)_\eps$ is
  Dirichlet-asymptotically fading (of order $k$), then the energy form
  $\qf d_{(X,g)}$ of $(X,g)$ and the (Dirichlet) energy form $\qf
  d^\Dir_{(X_\eps,g)}$ of $(X_\eps,g)$ with $X_\eps=X\setminus B_\eps$
  are $\delta_\eps$-quasi-unitarily equivalent of order $k$ with
  $\delta_\eps=\max\{\delta'_\eps,\Cellreg \delta'_\eps + \delta_\eps^+\}$.
\end{theorem}
%-----------------------------------------------------------------------
\begin{proof}
  We show again that the hypotheses\footnote{Note that the Dirichlet
    fading case is in some sense dual to the Neumann case, as here,
    $J^1$ needs a (more complicated) cut-off function and
    $J^{1\prime}$ is simply the extension by $0$.} of \Def{quasi-uni}
  are fulfilled, and specify the spaces and transplantation operators
  by
  \begin{align*}
    \map J {\HS:=\Lsqr{X,g}} {&\wt \HS:=\Lsqr{X_\eps,g}},
    & f & \mapsto f \restr {X_\eps},\\
    \map {J^1} {\HS^1:=\Sob{X,g}} {&\wt \HS^1:=\Sobn{X_\eps,g}},
    & f & \mapsto \chi_\eps f\\
   \map{J'}{\wt \HS=\Lsqr{X_\eps,g}}{&\HS=\Lsqr{X,g}},   
     & u & \mapsto \bar u,\\
   \map{J^{1\prime}}{\wt\HS^1=\Sobn{X_\eps,g}}{&\HS^1=\Sob{X,g}}, 
     &  u & \mapsto \bar u,
   \end{align*}
   where $\bar u$ denotes the extension of $\map u {X_\eps} \C$
   by $0$ on $B_\eps$.

   We check the hypotheses of \Def{quasi-uni}: We easily see that
   \begin{equation*}
     J'=J^*, \qquad
     JJ'=\id_{\wt\HS} \quadtext{and}
     J^{1\prime}=J' \restr{\wt \HS^1}.
   \end{equation*}
   As in the Neumann case, we have $\norm J=1$
   and~\eqref{eq:quasi-uni.a} is fulfilled with $\delta=0$.

   The first estimate in~\eqref{eq:quasi-uni.b} follows from the
   non-concentrating property \Defenum{dir-fading}{dir-fading.b},
   namely we have
  \begin{equation*}
    \norm[\Lsqr{X,g}] {f-J'Jf}
    = \norm[\Lsqr{B_\eps,g}] f
    \le \norm[\Lsqr{B_\eps^+,g}] f
    \le \delta'_\eps \norm[\Sob{X,g}] f.
  \end{equation*}
  Moreover, $Jf-J^1f=(1-\chi_\eps) f$, hence
  \begin{equation*}
    \norm[\Lsqr{X_\eps,g}]{Jf-J^1f}
    = \norm[\Lsqr{X_\eps,g}]{(1-\chi_\eps)f}
    \le \norm[\Lsqr{B_\eps^+ \cap X_\eps,g}] f
    \le \norm[\Lsqr{B_\eps^+,g}] f
    \le \delta'_\eps \norm[\Sob{X,g}] f
  \end{equation*}
  by the same argument.  Finally,
  \begin{align*}
    \bigabs{\qf d(f,J^{1\prime} u) - \qf d_\eps(J^1f,u)}
    &= \bigabs{\iprod[\Lsqr {T^*B_\eps^+,g}] {df - d(\chi_\eps f)} {d u}}\\
    &\le \bigabs{\iprod[\Lsqr {T^*B_\eps^+,g}] {(1-\chi_\eps) df} {d u}}
       + \bigabs{\iprod[\Lsqr {T^*B_\eps^+,g}] {f d \chi_\eps} {d u}}\\
    &\le \bigl(
            \norm[\Lsqr {T^*B_\eps^+,g}] {df}
            +\norm[\Lsqr{T^*B_\eps^+,g}]{fd\chi_\eps}
         \bigr)
         \norm[\Lsqr {T^*B_\eps^+,g}] {du}\\
    &\le \bigl(
           \Cellreg\delta'_\eps \norm{(\laplacian{(X,g)}+1)f} 
           + \delta_\eps^+ \norm{(\laplacian{(X,g)}+1)^{k/2}f}
         \bigr)
         \norm[\Sob {X_\eps,g}] u\\
    &=  (\Cellreg\delta'_\eps + \delta_\eps^+) \norm[k] f \norm[1] u
  \end{align*}
  by the non-concentrating property together with \Prp{non-concentr2}
  and the elliptic regularity of
  \Defenums{dir-fading}{dir-fading.b}{dir-fading.b'} and the moderate
  decay property \Defenum{dir-fading}{dir-fading.c}.
\end{proof}

%-----------------------------------------------------------------------
\subsection{Application: many small balls as Dirichlet obstacles}
\label{ssec:dir-fading}
%-----------------------------------------------------------------------
The obstacles are of the same kind as in \Subsec{neu-fading}.  Let
$I_\eps$ be $\eta_\eps$-separated as before (the results of this part
can be extended to \emph{uniformly locally finite} covers, see
\Def{unif.loc.cover}.  Let $\map{(\cdot)^+}{(0,r_0)}{(0,r_0)}$ be a
function such that $\eps \le \eps^+ \le \eta_\eps$ for all $\eps \in
(0,r_0)$.  Let
\begin{equation*}
  B_\eps^+:=B_{\eps^+}(I_\eps)=\bigcup_{p \in I_\eps} B_{\eps^+}(p).
\end{equation*}
We now check the conditions of \Def{dir-fading} and need good cut-off
functions.

Let us define the radially symmetric, harmonic function $h=h_m$ in
dimension $m$ given by
\begin{equation}
  \label{eq:def.h}
  h(s):= 
  \begin{cases}
    -\dfrac 1 {(m-2)s^{m-2}}, & m>2,\\
    \ln s,                   & m=2.
  \end{cases}
\end{equation}
Note that $h'(s)=1/s^{m-1}$.  Let us now define the radial cut-off
function $\map{\wt \chi_\eps}X{[0,1]}$ by
\begin{equation*}
  \wt \chi_\eps(r) =
  \begin{cases}
    0, & 0 \le r \le \eps,\\
    \dfrac{h(r)-h(\eps)}{h(\eps^+)-h(\eps)}, & \eps \le r \le \eps^+\\
    1, & \eps^+ \le r.
  \end{cases}
\end{equation*}
This function is Lipschitz-continuous.  We define the cut-off function
of \Def{dir-fading} by
\begin{equation}
  \label{eq:def.chi}
  \chi_\eps(x)
  := \wt\chi_\eps(d(x,p))
  \quadtext{for}
  x \in B_{\eta_\eps}(p)
\end{equation}
for each $p \in I_\eps$ and extend it by $1$ on $X \setminus
B_{\eta_\eps}$; again $\chi_\eps$ is Lipschitz-continuous.  Clearly,
$\supp(1-\chi_\eps)=B_\eps^+$ and $\chi_\eps \restr {B_\eps}=0$ by
definition.

%----------------------------------------------------------------------
\begin{remark*}
  For the moderate decay property of
  \Defenum{dir-fading}{dir-fading.c}, we need to control
  $\norm[\Lsqr{B_\eps^+,g}]{f d\chi_\eps}$ and will use Sobolev
  embedding theorems. If we stay in the $\Lsymb_2$-world, the order
  $k$ must satisfy $k >\dim X/2$ to have control of the
  $\Lsymb_\infty$-norm of $f$ by its $\Sobsymb^k$-norm, and we only
  need cut-off functions satisfying $\norm[\Lsqr{B_\eps^+,g}]{d
    \chi_\eps} \to 0$ as $\eps \to 0$.  The counterpart are strong
  assumptions concerning the sectional curvature to control the norm
  of $\Sobsymb^k$ with the graph norm in $\Sob[k]{\laplacian{(X,g)}}$:
  typically, one needs uniform bounds on the derivatives of the
  sectional curvature up to order $(k-2)$. Here, we prefer to stay at
  order $k=2$, using H\"older inequalities and the Sobolev embeddings
  given in \Prp{sob.est}\footnote{Indeed, there is an alternative
    given by Grigor'yan in~\cite{grigoryan}: the hypothesis of bounded
    geometry implies the Faber-Krahn inequality for small balls
    (Theorem~15.4), this Faber-Krahn inequality then implies an
    estimate on the heat kernel near $t=0$ (Corollary~15.6), namely
    $p_t(x,x)\leq c t^{-m/2}$ for $t\in(0,1)$ and finally this
    estimate gives that there exists $C>0$ such that $\sup_{x\in
      X}|((\Delta+1)^{-k}f)(x)|\leq C \norm[\Lsqr X] f$ for all $f \in
    \Lsqr X$ and all $k > m/4$ by Exercise~7.44 (p.~214).}  , this
  needs for the cut-off functions to satisfy
  $\norm[{\Lp[q]{B_\eps^+,g}}]{d \chi_\eps} \to 0$ as $\eps \to 0$ for
  some $q$, see \Prp{mod-decay-prop}.
\end{remark*}
%----------------------------------------------------------------------
Thus $q$ has to be as small as possible but the Sobolev embedding
forces $p$ to be not too large, at least for higher dimensions.  This
restriction leads us to introduce the following definition of $p_m$
and $q_m$ with $1/p_m+1/q_m=1$, namely
\begin{align*}
  p_m&=\frac{m}{m-4} \quad\text{if $m\geq 5$,}&
  p_4&=\frac 83, &
  p_3&=p_2=\infty,\\
  q_m&=\frac{m}{4} \hspace*{6.5ex}\text{if $m\geq 5$,}&
  q_4&=\frac 8 5,&
  q_3&=q_2=1.
\end{align*}

%----------------------------------------------------------------------
\begin{lemma}
  \label{lem:dir-fading.van-energy} 
  The cut-off function $\chi_\eps$ at a ball $B_{\eps^+}(p)$ satisfies
  \begin{equation*}
    \norm[{\Lp[2q_m]{T^*B_{\eps^+}(p),g}}]{d \chi_\eps}
    = \hat \delta_\eps
  \end{equation*}
  for all $p \in I_\eps$, where
  $\hat \delta_\eps=\Err(\eps^{(m-2q_m)/2q_m})$ if $m \ge 3$ resp.\
  $\hat \delta_\eps=\Err(1/\sqrt{\log (\eps^+/\eps)})$ if
  $m=2$.

  More precisely $\hat \delta_\eps=\Err(\eps)$ if $m \ge 5$,
  $\hat \delta_\eps=\Err(\eps^{1/4})$ if $m=4$, and
  $\hat \delta_\eps=\Err(\eps^{1/2})$ if $m=3$.
\end{lemma}
%-----------------------------------------------------------------------
\begin{proof}
  We calculate
  \begin{align*}
    \norm[{\Lp[2q_m]{T^*B_{\eps^+}(x),g}}]{d \chi_\eps}^{2q_m} 
    % \le K^{q_m+m/2} 
%    \norm[{\Lp[2q_m]{T^*B_{\eps^+}(x),g_\eucl}}]{d \chi_\eps}^{2q_m}\\
    &\le K^{q_m+m/2} \vol_{m-1} (\Sphere{m-1})
          \int_\eps^{\eps^+} \abs{\chi_\eps'(r)}^{2q_m} r^{m-1} \dd r\\
    &= K^{q_m+m/2} \frac{\vol_{m-1}(\Sphere{m-1})}{(h(\eps^+)-h(\eps))^{2q_m}}
          \int_\eps^{\eps^+} r^{(1-2q_m)(m-1)} \dd r\\
    &=: (\hat \delta_\eps)^{2q_m}
 \end{align*}
 by \Corenum{eucl.metric}{eucl.metric.b}.  If $m\neq 2$ the exponent
 of $r$ in the integral is different to $-1$, thus
 \begin{equation*}
   \hat \delta_\eps^{2q_m}
   =
   \begin{cases}
     K^{q_m+m/2} \dfrac{\vol_{m-1}(\Sphere{m-1})
           ((\eps^+)^{(m-2q_m(m-1))}-\eps^{m-2q_m(m-1)})}
         {(h(\eps^+)-h(\eps))^{2q_m}(m-2q_m(m-1))} &\text{if $m \ge 3$}\\[2ex]
     K^2 \dfrac{2\pi}{(\log\eps^+-\log\eps)} & \text{if $m=2$}
   \end{cases}
   \end{equation*}
  by the definition of $h$~\eqref{eq:def.h}.
\end{proof}
%-----------------------------------------------------------------------
We can now show the moderate decay property of
\Defenum{dir-fading}{dir-fading.c}:
%-----------------------------------------------------------------------
\begin{proposition}
  \label{prp:mod-decay-prop}
  Assume that $(X,g)$ is a complete manifold with bounded geometry and
  let $I_\eps$ be $\eta_\eps$-separated, then there exists
  $\delta_\eps^+$ such that
  \begin{equation*}
    \norm[\Lsqr{T^*B_\eps^+,g}]{f d\chi_\eps}
    \le \delta_\eps^+ \norm[{\Sob[2]{\laplacian{(X,g)}}}]{f}
  \end{equation*}
  for all $\eps>0$ with $\eps^+ \le \eta_\eps/4$ and $f \in \dom
  \laplacian{(X,g)}$, where
  \begin{equation*}
    \delta_\eps^+=
    \begin{cases}
      \Err(\eps/(\eps^+)^2) &\hbox{if $m\geq 5$}\\
      \Err(\eps^{1/4}/(\eps^+)^{5/4}) &\hbox{if $m=4$}\\
      \Err(\eps^{1/2}/(\eps^+)^{3/2}) &\hbox{if $m=3$}\\
      \Err(1/\eps^+\sqrt{\log (\eps^+/\eps)}) &\text{if $m=2$.}
    \end{cases}
  \end{equation*}
  In particular, if $\delta_\eps^+ \to 0$ as $\eps \to 0$, then
  \Defenum{dir-fading}{dir-fading.c} is fulfilled.
\end{proposition}
%-----------------------------------------------------------------------
\begin{proof}
  We have
  \begin{align*}
    \normsqr[\Lsqr{T^*B_\eps^+,g}]{f d\chi_\eps}
    &= \sum_{p \in I_\eps} \normsqr[\Lsqr{T^*B_{\eps^+}(p),g}]{f d\chi_\eps}\\
    &\le \sum_{p \in I_\eps} \normsqr[{\Lp[2p_m]{B_{\eps^+}(p),g}}] f
                       \normsqr[{\Lp[2q_m]{T^*B_{\eps^+}(p),g}}]{d\chi_\eps}\\
    &\le \CSob^2 (\eps^+)^{-2a_m} \hat \delta_\eps^2\sum_{p \in I_\eps} 
         \normsqr[{\Sob[2]{B_{4\eps^+}(p),g}}] f\\
    &\le \underbrace{\Cellreg^2\CSob^2 (\eps^+)^{-2a_m}\hat \delta_\eps^2}%
           _{=:(\delta_\eps^+)^2}
        \normsqr[{\Sob[2]{\laplacian{(X,g)}}}] f
  \end{align*}
  by H\"older's inequality for the first inequality, \Prp{sob.est} and
  \Lem{dir-fading.van-energy} for the second inequality and
  \Prp{ell.reg} for the last one.
\end{proof}
%-----------------------------------------------------------------------
Note that we have the integral estimate in \Lem{dir-fading.van-energy}
only for single balls, and used the supremum when considering all
balls in the previous proof.

Let us now set $\omega_\eps:=\eps/\eta_\eps$.
%-----------------------------------------------------------------------
\begin{theorem}
  \label{thm:dir-fading.balls}
  Let $(X,g)$ be a complete Riemannian manifold of bounded geometry,
  let $B_\eps=\bigdcup_{p \in I_\eps} B_\eps(p)$ be the union of balls
  of radius $\eps$ centred at the points of the $\eta_\eps$-separated
  set $I_\eps$ and put $\omega_\eps:=\eps/\eta_\eps$.  If there is
  $\gamma \in (0,1)$ such that
  \begin{subequations}
    \begin{align}
      \label{eq:dir-fading.a}
      \frac{\omega_\eps^{2\gamma}}{\eps} 
      &\to 0 &(m\geq 5),&&
      \frac{\omega_\eps^{5\gamma}}{\eps^4}
      &\to 0 &(m=4),\\
      \label{eq:dir-fading.b}
      \frac{\omega_\eps^{3\gamma}}{\eps^2}
      &\to 0 &(m=3),&&
      \frac{\omega_\eps^\gamma}{\eps \abs{\log \omega_\eps}^{1/2}} 
      &\to 0 &(m=2),
    \end{align}
  \end{subequations}
  then $(B_\eps)_\eps$ is
  Dirichlet-asymptotically fading, i.e., the energy form $\qf
  d_{(X,g)}$ and the (Dirichlet) energy form $\qf d^\Dir_{(X_\eps,g)}$
  are $\delta_\eps$-quasi-unitarily equivalent (of order $k=2$) with
  \begin{align*}
    \delta_\eps
    &=\Err \Bigl(
    \max \Bigl\{\omega_\eps^{1-\gamma}, 
                \frac{\omega_\eps^{2\gamma}}{\eps} \Bigr\} \Bigr)
      &(m \ge 5),&&
    \delta_\eps
    &= \Err \Bigl(\max \Bigl\{\omega_\eps^{1-\gamma}, 
                \frac{\omega_\eps^{5\gamma/4}}{\eps}\Bigr\} \Bigr),
      \quad (m =4),\\
    \delta_\eps
    &=\Err \Bigl(\max \Bigl\{\omega_\eps^{1-\gamma}, 
                \frac{\omega_\eps^{3\gamma/2}}{\eps}\Bigr\} \Bigr)
      &(m =3),&&
    \delta_\eps
    &=\Err \Bigl(\max \Bigl\{\omega_\eps^{1-\gamma} \abs{\log \omega_\eps}^{1/2}, 
       \frac{\omega_\eps^{\gamma}}{\eps \abs{\log \omega_\eps}^{1/2}}\Bigr\} \Bigr),
  \end{align*}
  where the latter $\delta_\eps$ is for the case $m=2$.
\end{theorem}
%-----------------------------------------------------------------------
\begin{proof}
  Conditions~\eqref{eq:dir-fading.a}--\eqref{eq:dir-fading.b} imply
  that $\omega_\eps=\eps/\eta_\eps \to 0$. We choose
  $\eps^+:=\eps\omega_\eps^{-\gamma}=\eps^{1-\gamma}\eta_\eps^\gamma$
  for some $\gamma \in (0,1)$, then also
  $\eps^+/\eta_\eps=\omega_\eps^{1-\gamma} \to 0$ and
  $\eps/\eps^+=\omega_\eps^\gamma \to 0$.  In particular, the
  non-concentrating property is fulfilled by \Prp{0} with error
  $\Err(\eps^+/\eta_\eps [\abs{\log
    \eps^+/\eta_\eps}])=\Err(\omega_\eps^{1-\gamma}[\abs{\log
    \omega_\eps}])$ (where $[\dots]$ appears only if $m=2$).
  Moreover, the moderate decay property is fulfilled once
  $\delta_\eps^+ \to 0$ by \Prp{mod-decay-prop}.
\end{proof}
%-----------------------------------------------------------------------
Let us give an example for $\eta_\eps$ such that $\delta_\eps \to 0$
as $\eps \to 0$ in the cases $m\geq 3$: If $\eta_\eps=\eps^\alpha$ for
some $\alpha \in(0,1)$, then $\omega_\eps=\eps^{1-\alpha}$ and
\begin{align*}
  \delta_\eps
  &=\Err \bigl(
  \max \bigl\{\eps^{(1-\alpha)(1-\gamma)}, \eps^{(1-\alpha)2\gamma-1}
       \bigr\} \bigr) 
  &&\text{for $m\geq 5$,}\\
  \delta_\eps
  &=\Err \bigl(\max \bigl\{\eps^{(1-\alpha)(1-\gamma)}, 
                \eps^{(1-\alpha)5\gamma/4-1}\bigr\} \bigr)
  &&\text{for $m=4$,}\\
  \delta_\eps
  &=\Err \bigl(\max \bigl\{\eps^{(1-\alpha)(1-\gamma)}, 
                \eps^{(1-\alpha)3\gamma/2-1}\bigr\} \bigr)
  &&\text{for $m=3$.}
\end{align*}
The condition on $\alpha, \gamma \in (0,1)$ for $\delta_\eps \to 0$ is
then for $m\geq 5$ that $\gamma > 1/(2(1-\alpha))$. We then need
$1/(2(1-\alpha))<1$, i.e., $ \alpha < 1/2$.

Similar estimates can be done in the other cases, we obtain $\alpha <
1/5$ for $m=4$ and $\alpha < 1/3$ for $m=3$.  For $m=2$ we need
$\eps^+ \sim\abs{\log \eps}^{-\alpha}$.

%-----------------------------------------------------------------------
\begin{remark*}
  Note that the critical parameter for the balls to fade is the
  \emph{capacity} (see the discussion in~\cite{rauch-taylor:75}
  or~\cite{khrabustovskyi-post:pre17}).  In our notation, the capacity
  of the balls of radius $\eps$ with $\eta_\eps$-separated balls is
  vanishing if $\eps^{m-2} \ll \eta_\eps^m$, i.e., if $\eps^{1-m/2}
  \ll \eta_\eps$ if $m \ge 3$, or $\abs{\log \eps}^{-1/2} \ll
  \eta_\eps$ if $m=2$.  If $\eta_\eps=\eps^\alpha$, then this means
  that $\alpha \in (0,1-2/m)$, and our above condition on $\alpha$ is
  only optimal for $m=3$ (then $\alpha <1/3$ as stated above), but too
  small if $m=4$ ($\alpha < 1/5$ instead of the optimal $\alpha <
  1/2$) and if $m \ge 5$ ($\alpha <1/2$ instead of the optimal $\alpha
  < 1-2/m$).  The opposite effect of solidifying happens if
  $\alpha>1-2/m$, see~\eqref{eq:spec.solid}.
\end{remark*}
%-----------------------------------------------------------------------

%-----------------------------------------------------------------------
%
\section{Solidifying obstacles for Dirichlet boundary conditions}
\label{sec:dir-solid}
%
%-----------------------------------------------------------------------

%-----------------------------------------------------------------------
\subsection{Abstract solidifying Dirichlet obstacles}
\label{sec:dir-solid-abstr}
%-----------------------------------------------------------------------

Let us now consider the case, when the obstacles fill out some closed
subset $S$, on which the limit operator has a Dirichlet boundary
condition (it ``solidifies'' on $S$).  We assume that the obstacles
$B_\eps$ in some sense ``converge'' to $S$ in the following sense:
%-----------------------------------------------------------------------
\begin{definition}
  \label{def:dir-solid}
  We say that a family $(B_\eps)_{\eps \in (0,\eps_0]}$ of closed subsets of a
  Riemannian manifold $(X,g)$ is \emph{Dirichlet-asymptotically
    solidifying} towards a closed subset $S$ if there is a sequence
  $(\chi_\eps)_\eps$ of Lipschitz-continuous cut-off functions
  $\map{\chi_\eps}X{[0,1]}$ with $\supp(\chi_\eps)\subset X_0:=X \setminus S$
  such that the following conditions are fulfilled: (we let $X_\eps:=X \setminus
  B_\eps$)
  \begin{enumerate}
  \item
    \label{dir-solid.b}
    \emph{Non-concentrating property:} We assume that
    $(A_\eps,X_\eps)$ is $\delta'_\eps$-non-concentrating of order
    $1$ with $\delta'_\eps \to 0$, and $(A_\eps,X_0)$ is
    $\delta''_\eps$-non-concentrating of order $2$ with
    $\delta''_\eps \to 0$, where $A_\eps:=\supp(d \chi_\eps)$ is an
    annulus region around the boundary of $S$.
  \item
    \label{dir-solid.b'}
    \emph{Elliptic regularity:} We assume that $(X_0,g)$ is elliptically
    regular, i.e., that there is $\Cellreg \ge 1$ such that
    \begin{equation*}
      \norm[{\Sob[2]{X_0,g}}] f
      \le \Cellreg \norm[\Lsqr{X_0,g}]{(\laplacianD{(X_0,g)}+1)f}
    \end{equation*}
    for all $f \in \Sob[2]{\laplacianD{(X_0,g)}}=\dom
    \laplacianD{(X_0,g)}$, where $\laplacianD{(X_0,g)}$ denotes the
    Dirichlet Laplacian on $(X_0,g)$.

  \item
    \label{dir-solid.a}
    \emph{Spectrally solidifying:} We assume $B_\eps \subset S$
    and that there is $\bar \delta_\eps \to 0$ as $\eps \to 0$ such
    that
    \begin{equation*}
      \norm[\Lsqr{\intr S \setminus \clo B_\eps,g}] u
      \le \bar \delta_\eps \norm[\Sob{X_\eps,g}] u
    \end{equation*}
    for all $u \in \Sobn{X_\eps,g}$ and $\eps \in (0,\eps_0]$.

  \item
    \label{dir-solid.c} 
    The cut-off functions $\chi_\eps$ have \emph{moderate decay} in
    the sense that
    \begin{equation*}
       \delta_\eps^+
       := \delta'_\eps \delta''_\eps \norm[\infty]{d \chi_\eps}
      \to 0
    \end{equation*}
    as $\eps \to 0$, where $\delta'_\eps$ and $\delta''_\eps$ are given
    in~\itemref{dir-solid.b}.
  \end{enumerate}
\end{definition}
%-----------------------------------------------------------------------
There is a subtle point in
\Defenums{dir-solid}{dir-solid.b}{dir-solid.c}: if we would assume
that $(A_\eps,X_0)$ is $\delta_\eps$-non-concentrating for the
\emph{same} $\delta_\eps=\delta'_\eps=\delta''_\eps$, then
$\delta_\eps^+$ will most likely not converge to $0$ as it contains
the cut-off function, see \Rem{dir-solid.c} for details.  This is why
we have two different assumptions of non-concentration in
\Defenum{dir-solid}{dir-solid.b}.

A sufficient condition for the spectral non-concentration property of
\Defenum{dir-solid}{dir-solid.a} is as follows (explaining also the
terminology) (Rauch-Taylor~\cite{rauch-taylor:75} say that such
obstacles ``become solid'' in $S$).

%CC-----------------------------------------------------------------------
\begin{proposition}
  Assume that $\lambda_\eps$ is the bottom of of the spectrum of the
  Laplacian on $\intr S \setminus \clo B_\eps$ with Dirichlet boundary
  conditions on $\bd B_\eps \setminus \bd S$ and Neumann boundary
  condition on $\bd S$.  If $\lim_{\eps \to 0} \lambda_\eps = \infty$,
  then $(B_\eps)_\eps$ is spectrally solidifying.
\end{proposition}
%-----------------------------------------------------------------------
\begin{proof}
  Note that the mentioned Laplacian is the operator associated with
  the quadratic form given by $\normsqr[\Lsqr{T^*(\intr S \setminus
    \clo B_\eps),g}]{du}$ with domain $u \in \set{f\restr {\intr S
      \setminus \clo B_\eps}}{f \in \Sobn{X_\eps}}$.  By the variational
  characterisation of the first eigenvalue, we have
  \begin{equation*}
    \lambda_\eps = \inf\BIGset{
      \frac{\int_{\intr S \setminus \clo B_\eps} \abssqr{du} \dd g}
           {\int_{\intr S \setminus \clo B_\eps} \abssqr u \dd g}}
         {u \in \Sobn {X_\eps} \setminus \{0\}}.
  \end{equation*}
  From this characterisation via an infimum, we conclude
  \begin{equation*}
    \norm[\Lsqr{\intr S \setminus \clo B_\eps,g}] u
    \le \frac 1 {\sqrt{\lambda_\eps}}
        \norm[\Lsqr{T^*(\intr S \setminus \clo B_\eps),g}] {du}
    \le \frac 1 {\sqrt{\lambda_\eps}}
        \norm[\Sob{X_\eps,g}] u.
  \end{equation*}
  As $\lambda_\eps \to \infty$, we can choose
  $\bar \delta_\eps=1/\sqrt{\lambda_\eps} \to 0$ as $\eps \to 0$.
\end{proof}
%-----------------------------------------------------------------------

Our next main result is as follows:
%-----------------------------------------------------------------------
\begin{theorem}
  \label{thm:dir-solid}
  Let $(X,g)$ be a Riemannian manifold and $(B_\eps)_\eps$ be a family
  of closed subsets of $X$.  If $(B_\eps)_\eps$ is
  Dirichlet-asymptotically solidifying towards $S$, then the Dirichlet
  energy form $\qf d^\Dir_{(X_0,g)}$ of $(X_0,g)$ with $X_0 = X
  \setminus S$ and the Dirichlet energy form $\qf d^\Dir_{(X_\eps,g)}$
  of $(X_\eps,g)$ with $X_\eps=X\setminus B_\eps$ are
  $\delta_\eps$-quasi-unitarily equivalent of order $2$ with
  $\delta_\eps=\max\{\bar \delta_\eps, \Cellreg (\delta''_\eps +
  \delta_\eps^+)\}$.
\end{theorem}
%-----------------------------------------------------------------------
\begin{proof}
  We show again that the hypotheses\footnote{Note that the Dirichlet
    solidifying case is in some sense dual to the Dirichlet fading
    case, as here, we have again $X_0 \subset X_\eps$, hence
    $J^{1\prime}$ is more complicated (as in the Neumann fading
    case).} of \Def{quasi-uni} are fulfilled.  Here, $X_0 \subset
  X_\eps$, so extension by $0$ and restriction are swapped.  We set
  \begin{align*}
    \map J {\HS:=\Lsqr{X_0,g}} {&\wt \HS:=\Lsqr{X_\eps,g}},
    & f & \mapsto \bar f,\\
    \map {J^1} {\HS^1:=\Sobn{X_0,g}} {&\wt \HS^1:=\Sobn{X_\eps,g}},
    & f & \mapsto \bar f\\
   \map{J'}{\wt \HS=\Lsqr{X_\eps,g}}{&\HS=\Lsqr{X_0,g}},   
     & u & \mapsto u \restr {X_0},\\
   \map{J^{1\prime}}{\wt\HS^1=\Sobn{X_\eps,g}}{&\HS^1=\Sobn{X_0,g}}, 
     &  u & \mapsto \chi_\eps u,
   \end{align*}
   where $\bar f$ denotes the extension of $\map f {X_0} \C$ by $0$
   onto $X_\eps$, as $X_0 \subset X_\eps$.

   We check the hypotheses of \Def{quasi-uni}: We easily see that
   \begin{equation*}
     J'=J^*, \qquad
     J'J=\id_\HS \quadtext{and}
     J^1=J \restr{\HS^1}.
   \end{equation*}
   As in the Neumann case, we have $\norm J=1$
   and~\eqref{eq:quasi-uni.a} is fulfilled with $\delta=0$.

   The second estimate in~\eqref{eq:quasi-uni.b} follows from the
   spectral non-concentrating property
   \Defenum{dir-solid}{dir-solid.a}, namely we have
  \begin{equation*}
    \norm[\Lsqr{X,g}] {u-JJ'u}
    = \norm[\Lsqr{\intr S \setminus \clo B_\eps,g}] u
    \le \bar \delta_\eps \norm[\Sob{X_\eps,g}] u.
  \end{equation*}
  Moreover, $J'u-J^{1\prime}u=((1-\chi_\eps) u)\restr{X_0}$, hence
  \begin{equation*}
    \norm[\Lsqr{X_0,g}]{J'u-J^{1\prime}u}
    = \norm[\Lsqr{X_0,g}]{(1-\chi_\eps)u}
    \le \norm[\Lsqr{A_\eps,g}] u
    \le \delta'_\eps \norm[\Sob{X_\eps,g}] u
  \end{equation*}
  by the non-concentration property of $(A_\eps, X_0)$ in
  \Defenum{dir-solid}{dir-solid.b} (implying the same property for
  $(A_\eps,X_\eps)$ as $X_0 \subset X_\eps$.  Finally,
  \begin{align*}
    \bigabs{\qf d_\eps(J^1f,u) - \qf d(f,J^{1\prime} u)}
    &= \bigabs{\bigiprod[\Lsqr {T^*A_\eps,g}] {df}{d((1-\chi_\eps) u)}}\\
    &\le \bigabs{\iprod[\Lsqr {T^*A_\eps,g}] {df}{(1-\chi_\eps) du}}
       + \bigabs{\iprod[\Lsqr {T^*A_\eps,g}] {df} {u\, d \chi_\eps}}\\
    &\le \norm[\Lsqr {T^*A_\eps,g}] {df}
         \bigl(
            \norm[\Lsqr {T^*A_\eps,g}] {du}
           +\norm[\Lsqr {A_\eps,g}] u \norm[\infty] {d\chi_\eps} 
         \bigr)\\
    &\le \delta''_\eps \norm[{\Sob[2]{X_0,g}}]f
         \bigl(1+\delta'_\eps \norm[\infty]{d\chi_\eps}\bigr)
         \norm[\Sob {X_\eps,g}] u\\
    &\le  \Cellreg (\delta''_\eps + \delta_\eps^+) 
    \norm{(\laplacianD{(X_0,g)}+1)f} \norm[1] u
  \end{align*}
  by the non-concentrating property of order $2$ in
  \Defenum{dir-solid}{dir-solid.b} for the second last estimate and
  the elliptic regularity and the moderate decay property
  (\Defenums{dir-solid}{dir-solid.b'}{dir-solid.c}) for the last
  estimate.
\end{proof}
%-----------------------------------------------------------------------

%-----------------------------------------------------------------------
\subsection{Application: many solidifying small balls as Dirichlet
  obstacles}
\label{ssec:neu-solid}
%-----------------------------------------------------------------------
The obstacles are of the same kind as in \Subsec{neu-fading} but
denser: let now $I_\eps$ be $\eps$-separated and let
$B_\eps=\bigcup_{p \in I_\eps} B_\eps(p)$ be the disjoint union of
balls of radius $\eps$.  Before checking the conditions of
\Def{dir-solid}, we first need the following result:
% -----------------------------------------------------------------------
\begin{lemma}[Rauch-Taylor~\cite{rauch-taylor:75}]
  \label{lem:lambda}
  Assume that $\eta>\eps$ and that
  \begin{equation*}
    A_{\eps,\eta}(0):= B_\eta(0) \setminus \clo {B_\eps(0)}
  \end{equation*}
  is an annulus with inner radius $\eps$ and outer radius $\eta$ in
  Euclidean space $\R^m$.  Denote by $\lambda^\eucl_\eps$ the first
  eigenvalue of the Laplacian with Dirichlet boundary condition on the
  inner sphere, and Neumann on the outer sphere.  Then there exists a
  constant $C_\eucl>0$ (depending only on the dimension) such that
  \begin{equation*}
    \lambda^\eucl_\eps
    \ge \frac{C_\eucl\eps^{m-2}}{\eta^m}
    \quad\text{for $m \ge 3$ \qquad resp.}\quad
    \lambda^\eucl_\eps
    \ge \frac{C_\eucl}{\eta^2\abs{\log\eps}}\quad\text{for $m=2$.}
  \end{equation*} 
  for all $0<\eps<\eta < r_0$.
\end{lemma}
%-----------------------------------------------------------------------

%-----------------------------------------------------------------------
\begin{definition}
  \label{def:unif.loc.cover}
  We say that $\{B_{\eta_\eps}(p)\}_{p \in I_\eps}$ is a
  \emph{uniformly locally finite cover} of $S$ if there is $\eps_0>0$
  and $N \in \N$ such that
  \begin{equation}
    \label{eq:uni.loc.bdd}
    \card {\bigset{q \in I_\eps}
          {B_{\eta_\eps}(p) \cap B_{\eta_\eps}(q) \ne \emptyset}}
    \le N
    \qquadtext{and} 
    S\subset B_{\eta_\eps}=\bigcup_{p\in I_\eps}B_{\eta_\eps}(p)
  \end{equation}
  for all $q \in I_\eps$ and all $\eps \in (0,\eps_0]$.
\end{definition}
%-----------------------------------------------------------------------

%-----------------------------------------------------------------------
\begin{proposition}
  \label{prp:lambda}
  Assume that $(X,g)$ is a Riemannian manifold with bounded geometry
  with harmonic radius $r_0>0$.  Let $\eps, \eta_\eps \in (0,r_0)$
  such that $0<\eps<\eta_\eps < r_0$.  Assume that $I_\eps$ is
  $\eps$-separated and that $(B_{\eta_\eps}(p))_{p \in I_\eps}$ is a
  uniformly locally finite cover of $S$.

  Then we have
  \begin{equation}
    \label{eq:spec.non.conc.b}
    \norm[\Lsqr{\intr S \setminus \clo B_\eps,g}] u 
    \le\norm[\Lsqr{A_{\eps,\eta_\eps},g}] u 
    \le \bar \delta_\eps \norm[\Sob{X_\eps,g}] u
  \end{equation}
  for all $u \in \Sob{X_\eps,g}$, where
  $A_{\eps,\eta_\eps}=B_{\eta_\eps} \setminus \clo B_\eps$
  and
  \begin{align*}
    \bar \delta_\eps
    &=C \sqrt{\eta_\eps^m/\eps^{m-2}} \quad(m \ge 3) 
    &\text{resp.}\qquad 
    \bar \delta_\eps
    &=C\eta_\eps\sqrt{\abs{\log\eps}} \quad (m=2)
  \end{align*}
  for some constant $C>0$ depending only on $N$, $K$ and $m$.  In
  particular, if $\eta_\eps^m/\eps^{m-2} \to 0$ as $\eps \to 0$ then
  $(B_\eps)_\eps$ is spectrally solidifying (see
  \Defenum{dir-solid}{dir-solid.a}).
\end{proposition}
%-----------------------------------------------------------------------
\begin{proof}
  Note first that $\intr S \setminus \clo B_\eps \subset
  A_{\eps,\eta_\eps}$, hence we have
  \begin{align*}
    \normsqr[\Lsqr{\intr S \setminus \clo B_\eps,g}] u 
    \le\norm[\Lsqr{A_{\eps,\eta_\eps},g}] u 
    &\le \sum_{p \in I_\eps} \normsqr[\Lsqr{A_{\eps,\eta_\eps}(p),g}] u \\
    &\le \frac{K^{m+1}}{C_\eucl} \cdot \frac {\eta_\eps^m}{\eps^{m-2}}
      \sum_{p \in I_\eps} \normsqr[\Lsqr{T^*A_{\eps,\eta_\eps}(p),g}] {du}\\
    &\le \underbrace{\frac{N K^{m+1}}{C_\eucl}}_{=:C^2}  
            \cdot \frac {\eta_\eps^m}{\eps^{m-2}}
      \normsqr[\Lsqr{T^*A_{\eps,\eta_\eps},g}] {du}
  \end{align*}
  using \Corenum{eucl.metric}{eucl.metric.b} and \Lem{lambda}, where
  $A_{\eps,\eta}(p):=\intr B_\eta(p) \setminus \clo B_\eps(p)$ is the
  annulus with inner radius $\eps$ and outer radius $\eta$ around $p$
  and $A_{\eps,\eta}:= \bigcup_{p \in I_\eps}A_{\eps,\eta}(p)$.
\end{proof}
%-----------------------------------------------------------------------
If $\eta_\eps=\eps^\alpha$ with $\alpha \in (0,1)$, then $B_\eps$ is
spectrally solidifying if
\begin{equation}
  \label{eq:spec.solid}
  \frac {m-2}m < \alpha.
\end{equation}

To check the remaining properties of \Def{dir-solid} we need some
regularity on $Y=\bd S$.
%-----------------------------------------------------------------------
\begin{assumption}[Geometric asumption on the boundary of the
  solidifying set]
  \label{ass:regtub}
  We assume that $Y=\bd S$ is a smooth manifold with embedding
  $\embmap \iota Y X$ and induced metric $h:=\iota^* g$, we assume
  also that \emph{$Y$ admits a uniform tubular neighbourhood}, i.e.,
  that $Y$ has a global normal unit vector field $\vec N$ (so $Y$
  is orientable) and that there is $r_0>0$ such that
  \begin{equation}
    \label{eq:expnormal}
    \map{\exp}{Y\times [0,r_0)} X,
    \qquad
    (y,t) \mapsto \exp_y(t\vec N(y))
  \end{equation}
  is a diffeomorphism.
\end{assumption}
%-----------------------------------------------------------------------

%-----------------------------------------------------------------------
\begin{remark}
  \label{rem:regtub}
  This assumption includes the fact that the principal curvatures of
  the hypersurface $Y$ are bounded by a constant depending on $1/r_0$
  and $\kappa_0$, see e.g.~\cite[Cor.~3.3.2]{H-K}.  But our assumption
  is stronger: we need also that $Y$ does not admit infinitely close
  points which are far away for the inner distance.
\end{remark}
%-----------------------------------------------------------------------

Let $\wt \eps \in (0,r_0)$ be a function of $\eps$ such that $\wt \eps
\to 0$ as $\eps \to 0$ (to be specified later).  Moreover set
\begin{equation*}
  A_{\wt\eps} := \set{x \in X_0=X \setminus S}{d(x,S) < \wt \eps}.
\end{equation*}
Then $A_{\wt \eps}$ has tubular coordinates $(r,y) \in (0,\wt \eps)
\times Y$ by \Ass{regtub}.

Let $\map{\wt \chi} \R {[0,1]}$ be a smooth function with $\wt
\chi(r)=0$ for $r\le 0$, $\chi$ strictly monotone on $(0,1)$ and $\wt
\chi(r)=1$ for $r \ge 1$ and $\norm[\infty]{\wt \chi'}\le 2$.  We then
define
\begin{equation}
  \label{eq:def.chi2}
  \chi_{\wt \eps}(x) 
  := \wt \chi\Bigl(\frac{d(x,S)}{\wt \eps}\Bigr)
\end{equation}
as cut-off function.  We clearly have $\norm[\infty]{d \chi_{\wt
    \eps}}\le 2/\wt \eps$ and $A_{\wt \eps} = \supp(d \chi_{\wt \eps})
\cap X_0$

% Olaf: I think we can replace the proposition and remark by the
% tubular neighbourhood estimate of order 1 ...  NOOO! I just read
% your comment in Rem. dir-solid.c!!!
%-----------------------------------------------------------------------
\begin{proposition}
  \label{prp:dir-solid.balls.b0}
  Assume that $(X,g)$ has bounded geometry with harmonic radius $r_0>0$.
  Assume additionally that
  \begin{equation}
    \label{eq:ass.regtub}
    A_{\wt\eps} \subset B_{\eta_\eps}
  \end{equation}
  (it then follows that $A_{\wt\eps} \subset B_{\eta_\eps} \setminus
  \clo B_\eps$) and that~\eqref{eq:spec.non.conc.b} holds.  Then
  \begin{equation*}
    \norm[\Lsqr{A_{\wt\eps},g}] u
    \le \bar \delta_\eps \norm[\Sob{X_\eps,g}] u
  \end{equation*}
  for all $u \in \Sob{X_\eps,g}$ and $\wt \eps \in (0,r_0)$ ($\bar
  \delta_\eps$ is given in \Prp{lambda}).  In particular, $(A_{\wt
    \eps},X_\eps)$ is $\bar \delta_\eps$-non-concentrating of order
  $1$.
\end{proposition}
%-----------------------------------------------------------------------
\begin{proof}
  As $A_{\wt \eps} \subset A_{\eps,\eta_\eps}=B_{\eta_\eps} \setminus
  \clo B_\eps$, we have
  \begin{equation*}
    \norm[\Lsqr{A_{\wt \eps},g}] u
    \le \norm[\Lsqr{A_{\eps,\eta_\eps},g}] u
    \le \bar \delta_\eps \norm[\Sob{X_\eps,g}] u
  \end{equation*}
  using~\eqref{eq:spec.non.conc.b}.
\end{proof}
%-----------------------------------------------------------------------
\begin{remark}
  \label{rem:dir-solid.b}
  Note that there is a hidden assumption on $\wt \eps$ and $\eta_\eps$
  in $A_{\wt \eps} \subset B_{\eta_\eps}$: namely, as $A_{\wt \eps}$ is the 
  $\wt \eps$-neighbourhood of $S$ and $B_{\eps} \subset S$, such an
  inclusion can only be true if $\wt \eps/\eta_\eps$ tends to $0$ or
  at least is bounded.
\end{remark}
%-----------------------------------------------------------------------

%-----------------------------------------------------------------------
\begin{proposition}
  \label{prp:dir-solid.balls.b1}
  Assume that $(X,g)$ has bounded curvature with radius $r_0>0$.
  Assume additionally that $(Y,h)$ is a complete smooth orientable
  hypersurface admitting a uniform tubular neighbourhood also with
  radius $r_0>0$.  Then there is a constant $C'>0$ depending only on
  $Y$ and $r_0$ such that
  \begin{equation*}
    \norm[\Lsqr{A_{\wt\eps},g}] {df} 
    \le C' \sqrt{\wt \eps} \norm[{\Sob[2]{X_0,g}}] f
  \end{equation*}
  for all $f \in \Sob[2]{X_0,g}$ and $\wt \eps \in (0,r_0)$.  In
  particular, $(A_{\wt \eps},X_0)$ is
  $C' \sqrt{\wt \eps}$-non-concentrating of order $2$.
\end{proposition}
%-----------------------------------------------------------------------
\begin{proof}
  From \Lem{0-hypersurface} (with $\wt\eps$ and $r_0$ instead of
  $\eps$ and $\eps^+$) we conclude that $(A_{\wt \eps}, X_0)$ is
  $C'\sqrt{\wt \eps}$-non-concentrating, and \Prp{non-concentr2} then
  yields
  \begin{equation*}
    \norm[\Lsqr{A_{\wt\eps},g}] {df}
    \le C' \sqrt{\wt \eps} \norm[{\Sob[2]{A_{r_0},g}}] f
    \le C' \sqrt{\wt \eps} \norm[{\Sob[2]{X_0,g}}] f
  \end{equation*}
  for all $f \in \Sob[2]{X_0,g}$.
\end{proof}
%-----------------------------------------------------------------------

%-----------------------------------------------------------------------
\begin{corollary}
  \label{cor:dir-solid.c}
  Assume that $\eta_\eps^m/\eps^{m-2} \to 0$
  (resp. $\eta_\eps^2\abs{\log\eps} \to 0$), then the cut-off function
  $\chi_{\wt \eps}$ has moderate decay, i.e., there is $\wt \eps \in
  (0,r_0)$ with $\wt \eps \to 0$ such that
  \Defenum{dir-solid}{dir-solid.c} is fulfilled.
\end{corollary}
%-----------------------------------------------------------------------
\begin{proof}
  Let $\wt \eps := (\eta_\eps^m/\eps^{m-2})^\gamma$ if $m \ge 3$
  (resp.\ $\wt \eps := (\eta_\eps^2\abs{\log\eps})^\gamma$ if $m=2$)
  for some $\gamma \in (0,1)$.  Set $\delta''_\eps:=C'\sqrt{\wt \eps}$,
  then we have
  \begin{equation*}
    \delta_\eps^+
    = \bar \delta_\eps \delta''_\eps \norm[\infty]{d\chi_\eps}
    = 2 C  C'
    \Bigl(\frac {\eta_\eps^m}{\wt \eps \eps^{m-2}} \Bigr)^{1/2}
    = C  C'
    \Bigl(\frac {\eta_\eps^m}{\eps^{m-2}} \Bigr)^{(1-\gamma)/2}
  \end{equation*}
  as $\norm[\infty]{d\chi_\eps} \le 2/\wt \eps$, and hence $\delta_\eps^+
  \to 0$ as $\eps \to 0$.  A similar argument holds for $m=2$.
\end{proof}
%-----------------------------------------------------------------------

%-----------------------------------------------------------------------
\begin{remark}
  \label{rem:dir-solid.c}
  There is a subtle point in the combination of arguments for the
  non-concentrating property: If we used for \Prp{dir-solid.balls.b1}
  an analogue result as for \Prp{dir-solid.balls.b0} (with
  $\delta'_\eps$ instead of $\bar \delta_\eps$ also of order
  $\sqrt{\wt \eps}$), then $\delta_\eps^+$ would not tend to $0$, as
  $\delta'_\eps \delta''_\eps$ is of order $\wt \eps$, but
  $\norm[\infty]{d\chi_\eps}$ is of order $\wt \eps^{-1}$.  So we need
  somehow also $\intr S \setminus B_\eps$ for the convergence.  In
  particular, we need that $A_\eps$ is covered by $B_{\eta_\eps}$,
  which assures that the balls in $B_\eps$ are not too far separated,
  see \Rem{dir-solid.b}.  This is also the reason why we need the
  additional regularity on $\bd S$ in \Ass{regtub}.
\end{remark}
%-----------------------------------------------------------------------

We can now state our main result of solidifying of a union of many
balls:
%-----------------------------------------------------------------------
\begin{theorem}
  \label{thm:dir-solid.balls}
  Let $(X,g)$ be a complete Riemannian manifold of bounded geometry
  with harmonic radius $r_0>0$ and let $B_\eps=\bigdcup_{p \in I_\eps}
  B_\eps(p)$ be the union of $\eps$-separated balls of radius $\eps$.
  Assume that there is $\eta_\eps \in (0,r_0)$ such that the following
  holds:
  \begin{enumerate}
  \item 
    \label{dir-solid.balls.a}
    there is a closed subset $S \subset X$ with smooth boundary $Y=\bd
    X$ admitting a uniform tubular neighbourhood of radius $r_0>0$;
    denote by $A_{\wt\eps}$ the (outer) $\wt \eps$-neighbourhood,
    where $\wt \eps=(\eta_\eps^m/\eps^{m-2})^\gamma$ (resp $\wt \eps
    := (\eta_\eps^2\abs{\log\eps})^\gamma$ if $m=2$) for some $\gamma
    \in (0,1)$.
  \item 
    \label{dir-solid.balls.b}
    we have
    \begin{equation*}
      B_\eps \subset S \quadtext{and}
      A_{\wt\eps} \subset B_{\eta_\eps}
    \end{equation*}
    and the latter cover $(B_{\eta_\eps})_{p \in I_\eps}$ is uniformly
    locally bounded (see~\eqref{eq:uni.loc.bdd}).
  \item 
    \label{dir-solid.balls.c}
    We have
    \begin{equation*}
      \frac {\eta_\eps^m}{\eps^{m-2}} \to 0 \quad (m \ge 3) \quadtext{resp.}
      \eta_\eps^2\abs{\log\eps}\to 0 \quad(m = 2) \qquad \text{as $\eps \to 0$.}
    \end{equation*}
  \item 
    \label{dir-solid.balls.d}
    Finally, 
    \begin{equation*}
      \frac{\wt \eps}{\eta_\eps}
      = \frac{\eta_\eps^{m\gamma-1}}{\eps^{\gamma(m-2)}} \quad (m \ge 3)
      \quadtext{resp.}
      \eta_\eps^{2\gamma-1}\abs{\log\eps} \quad (m = 2)
    \end{equation*}
    is bounded as $\eps \to 0$.
  \end{enumerate}
  Then $(B_\eps)_\eps$ is Dirichlet-asymptotically solidifying towards
  $S$, i.e., the Dirichlet energy form $\qf d^\Dir_{(X_0,g)}$ and the
  Dirichlet energy form $\qf d^\Dir_{(X_\eps,g)}$ are
  $\delta_\eps$-quasi-unitarily equivalent with
  \begin{equation*}
    \delta_\eps
    = \Err\Bigl(\frac{\eta_\eps^m}{\eps^{m-2}}\Bigr)^{(1-\gamma)/2}
    \quad(m \ge 3)
    \quadtext{resp.}
    \Err\Bigl(\eta_\eps^2\abs{\log\eps}\Bigr)^{(1-\gamma)/2}
    \quad (m=2).
  \end{equation*}
  Here, $X_\eps = X \setminus B_\eps$ and $X_0 = X \setminus S$.
\end{theorem}
%-----------------------------------------------------------------------
\begin{proof}
  For the elliptic regularity \Defenum{dir-solid}{dir-solid.b'} we remark
  that the proof of \Prp{ell.reg} based on \Prp{ell.reg2} works
  as well for the Dirichlet-Laplacian.
  Combine now \Prp{lambda}, \Prps{dir-solid.balls.b0}{dir-solid.balls.b1}
  and \Cor{dir-solid.c} and apply \Thm{dir-solid} with
  \begin{equation*}
    \delta_\eps
    = \max \bigl\{ 
      \bar \delta_\eps,
      \Cellreg(\delta''_\eps+\delta_\eps^+)
    \bigr\}.
  \end{equation*}
  The error is then (for $m\geq 3$) of order
  \begin{equation*}
    \max \Bigl\{
       \Err\Bigl(\frac{\eta_\eps^m}{\eps^{m-2}}\Bigr)^{1/2},
       \Err\Bigl(\frac{\eta_\eps^m}{\eps^{m-2}}\Bigr)^{\gamma/2},
       \Err\Bigl(\frac{\eta_\eps^m}{\eps^{m-2}}\Bigr)^{(1-\gamma)/2}
       \Bigr\},
  \end{equation*}
  so the error term is dominated by the middle term if $\gamma \in
  [1/2,1)$.
\end{proof}
%-----------------------------------------------------------------------
Let us give an example for the case $m\geq 3$: Let
$\eta_\eps=\eps^\alpha$ then in order to
have~\Thmenum{dir-solid.balls}{dir-solid.balls.c} and $\eps <
\eta_\eps$, we need
\begin{equation}
  \label{eq:alpha}
  \frac{m-2}m < \alpha \le 1.
\end{equation}
For~\Thmenum{dir-solid.balls}{dir-solid.balls.d} to be true, we then need
\begin{equation*}
  \frac{\wt \eps}{\eta_\eps}
  = \Bigl(\frac{\eta_\eps^m}{\eps^{m-2}}\Bigr)^\gamma \frac 1{\eta_\eps}
  = \frac{\eta_\eps^{m \gamma - 1}}{\eps^{\gamma(m-2)}}
  = \eps^{\alpha(m \gamma - 1)- \gamma(m-2)}
  = \eps^{\alpha(m \gamma - 1)- \gamma(m-2)}
\end{equation*}
to be bounded, i.e.,
\begin{equation*}
  \alpha(m\gamma-1)-\gamma(m-2)\ge 0,
  \quadtext{or equivalently}
  \gamma \ge \frac \alpha{m\alpha-(m-2)}.
\end{equation*}
On the other hand $\gamma \in (0,1)$, so we have the necessary condition
\begin{equation*}
  1 >  \frac \alpha{m\alpha-(m-2)},
  \quadtext{or equivalently}
  \alpha > \frac{m-2}{m-1},
\end{equation*}
which is a stricter condition than~\eqref{eq:alpha}.  Note that the
condition in \Thmenum{dir-solid.balls}{dir-solid.balls.d} is needed in
order to have a uniformly locally finite cover.

\appendix 
%-----------------------------------------------------------------------
%
% 
\section{Sobolev estimates on balls on manifolds}
\label{app:sob.est}
%
%-----------------------------------------------------------------------

%-----------------------------------------------------------------------
\begin{proposition}
  \label{prp:sob.est}
  Assume that $(X,g)$ is complete and has bounded geometry with
  harmonic radius $r_0>0$.  Then there is a constant $\CSob >0$ such
  that
  \begin{align*}
    \norm[{\Lp[2p_m]{B_r(x),g}}] f
    \le \CSob \, r^{-a_m}\, \norm[{\Sob[2]{B_{4r}(x),g}}] f
  \end{align*}
  % \begin{align*}
  %   m\geq 5\quad &\norm[{\Lp[\frac{2m}{m-4}]{B_r(x),g}}] f
  %   \le \CSob\, r^{-2}\, \norm[{\Sob[2]{B_{4r}(x),g}}] f\\
  %   m=4\quad &\norm[{\Lp[16/3]{B_r(x),g}}] f
  %   \le \CSob\, r^{-5/4}\, \norm[{\Sob[2]{B_{4r}(x),g}}] f\\
  %   m=3\quad &\norm[{\Lp[\infty]{B_r(x),g}}] f
  %   \le \CSob\, r^{-3/2}\, \norm[{\Sob[2]{B_{4r}(x),g}}] f\\
  %   m\geq 2\quad &\norm[{\Lp[\infty]{B_r(x),g}}] f
  %   \le \CSob\, r^{-1}\, \norm[{\Sob[2]{B_{4r}(x),g}}] f
  % \end{align*}
  for all $x \in X$, $r \leq r_0/4$ and $f \in \Sob[2]{B_{4r}(x),g}$,
  where
  \begin{align*}
    a_m& =2 \quad (m \ge 5),&
    a_4& = 5/4,\quad &
    a_3& = 3/2, \quad&
    a_2& = 1.
  \end{align*}
\end{proposition}
%-----------------------------------------------------------------------
\begin{proof}
  The Sobolev embedding theorem in $\R^m$ states that
  % for $m\geq 4$ and $p \leq \frac{2m}{m-2}$
  $\Sobx[1]{q}{\R^m}\subset\Lp[p]{\R^m}$ is a continuous embedding
  provided $1/p=1/q-1/m$ ({\cite[Thm2.5]{hebey2}}).  Thus, using a
  cut-off function we conclude that there exists a constant
  $C_{p,q}>0$ such that
  \begin{equation*}
    \norm[{\Lp[p]{B_1(0),g_\eucl}}] f 
    \le C_{p,q}\, \norm[{\Sobx[1]{q}{B_2(0),g_\eucl}}] f
  \end{equation*}
  for all $f\in \Sobx[1]{q}{\R^m}$.  By a scaling argument we
  conclude that
  \begin{align*}
    \norm[{\Lp[p]{B_r(0),g_\eucl}}] f 
    &\le \frac{C_{p,q}}{2^{m/q}} r^{m(\frac 1 p-\frac 1 q)}
      \norm[{\Sobx[1]{q}{B_{2r}(0),g_\eucl}}] f\\
    &\le \frac{C_{p,q}}{2^{m/q}} r^{-1}
      \norm[{\Sobx[1]{q}{B_{2r}(0),g_\eucl}}] f
  \end{align*}
  for all $f \in \Sobx[1]{q, \loc}{\R^m}$.  Finally, by the hypothesis
  of bounded geometry, we obtain
  \begin{equation}
    \label{eq:sobpq}
    \norm[{\Lp[p]{B_r(x),g}}] f 
    \le C(p,q,K) r^{-1}
    \norm[{\Sobx[1]{q}{B_{2r}(x),g}}] f
  \end{equation}
  for all $f \in \Sobx[1]{q, \loc}{X,g}$ and $x \in X$ as soon as
  $2r\leq r_0$. To obtain the desired estimate we have to apply this
  kind of control twice.

  If $m\geq 5$, let $p$ and $p'$ be such that 
  \begin{equation*}
    \frac 1 {p'}=\frac 1 2-\frac 1 m
    \quadtext{and}
    \frac 1 p=\frac 1 {p'}-\frac 1 m,
    \quadtext{thus}
    \frac 1 p=\frac 1 2-\frac 2 m=\frac{m-4}{2m}.
  \end{equation*}
  Let $f\in \Sobx[2]{2}{X,g}$, and $r\leq r_0/4$. We know already that
  \begin{equation*}
    \norm[{\Lp[p]{B_r(x),g}}] f 
    \le C(p,q,K)\, r^{-1}
    \norm[{\Sobx[1]{p'}{B_{2r}(x),g}}] f.
  \end{equation*}
  Moreover, applying~\eref{eq:sobpq} to the function $\phi=\abs{df}$
  we obtain
  \begin{equation*}
    \norm[{\Lp[p']{B_{2r}(x),g}}] {df} 
    \le C(p',2,K)\, r^{-1}
    \norm[{\Sob[1]{B_{4r}(x),g}}] {\,|df|\,}
  \end{equation*}
  We now argue as in~\eqref{eq:non-concentr3} and estimate $\abs[g]
  {d\phi}\le \abs[g]{\nabla^2 f}$, hence we have
  \begin{equation*}
    \norm[{\Lp[p]{B_r(x),g}}] f 
    \le C(p,K)\,r^{-2}\norm[{\Sob[2]{B_{4r}(x),g}}]f
  \end{equation*}
  for all $f\in \Sobx[2]{2}{X,g}$ and $x \in X$ with $C(p,K)=C(p',2,K)
  C(p,p',K)$.  For small dimensions, we can use the following special
  Sobolev imbeddings results: there exists a constant $C>0$ such that
  \begin{align}
    \label{eq:qsupm}
    \norm[{\Lp[\infty]{B_{1}(0)}}]f
    &\leq C\norm[{\Sobx[1]{q}{B_{2}(0)}}]f,&
    \norm[{\Lp[\infty]{B_{r}(0)}}]f 
    &\leq r^{-m/q}C\norm[{\Sobx[1]{q}{B_{2r}(0)}}]f\\
    \label{eq:qeqm}%q=m &\Rightarrow 
    \norm[{\Lp[\frac{km}{m-1}]{B_{1}(0)}}]f
    &\leq C\norm[{\Sobx[1]{m}{B_{2}(0)}}]f,&
    \norm[{\Lp[\frac{km}{m-1}]{B_{r}(0)}}]f 
    &\leq r^{(m-1-k)/k}C\norm[{\Sobx[1]{m}{B_{2r}(0)}}]f
  \end{align}
  for all $q>m$ and all $f \in \Sobx[1]{q}{B_{2}(0),g_\eucl}$, where
  $k\in \N \setminus \{0\}$, see~\cite[Thm.~2.7]{hebey2}
  and~\cite[Thm.~1.4.4]{saloff}.

  For $m=4$, choose $p'=4$ and $p=\frac{4 \cdot 4}{4-1}=\frac{16}3$,
  then we have, applying~\eqref{eq:qeqm} with $k=4$ and using the
  assumption of bounded geometry,
  \begin{equation*}
    \norm[{\Lp[16/3]{B_r(x),g}}] f 
    \le C(8/3,K)\,r^{-5/4}\norm[{\Sob[2]{B_{4r}(x),g}}]f
  \end{equation*}
  for all $f\in \Sobx[2]{2}{X,g}$ and $x \in X$

  For $m=3$, choose $p'=6$ and $p=\infty$, then we have,
  applying~\eqref{eq:qsupm} using the assumption of bounded geometry,
  \begin{equation*}
    \norm[{\Lp[\infty]{B_r(x),g}}] f 
    \le C(\infty,K)\,r^{-3/2}\norm[{\Sob[2]{B_{4r}(x),g}}]f.
  \end{equation*}
  for all $f \in \Sobx[2]{2}{X,g}$ and $x \in X$.

  Finally, for $m=2$, choose $p'=4$ and $p=\infty$, then
  \begin{equation*}
    \norm[{\Lp[\infty]{B_r(x),g}}] f 
    \le C(\infty,K)\,r^{-1}\norm[{\Sob[2]{B_{4r}(x),g}}]f
  \end{equation*}
  for all $f\in \Sobx[2]{2}{X,g}$ and $x\in X$.
\end{proof}
%-----------------------------------------------------------------------

%-----------------------------------------------------------------------
\begin{lemma}
  \label{lem:0-hypersurface}
  Assume that $(X,h)$ has bounded geometry with harmonic radius
  $r_0>0$ and that $(Y,h)$ is a complete orientable submanifold of
  codimension $1$ in $X$ (a \emph{hypersurface}).  We assume that $Y$
  admits a uniform tubular neighbourhood (as defined in \Ass{regtub})
  also with radius $r_0>0$

  Let $\eps$ and $\eps^+$ such that $0<\eps <\eps^+ < r_0 \le 1$.
  Then there is $C>0$ depending only on $Y$ and $r_0$ such that
  \begin{equation*}
    \norm[\Lsqr{B_\eps(Y),g}] f
    \le C\Bigl(\frac \eps {\eps^+} \Bigr)^{1/2}
    \norm[\Sob{B_{\eps^+}(Y),g}] f
  \end{equation*}
  for all $f \in \Sob{X,g}$.
\end{lemma}
%-----------------------------------------------------------------------
\begin{proof}
  In the coordinates defined by $\exp$ in~\eqref{eq:expnormal} the
  metric is of the form $dt^2+h(t)$ where $h(t)$ is metric on $Y$
  equal to $h$ at $t=0$.  We then apply~\cite[Lem.~A.2.16]{post:12}
  with $a=\eps$ and $b=\eps^+$ and obtain that $([0,\eps]\times
  Y,[0,\eps^+]\times Y)$ is $2(\eps/\eps^+)$-non-concentrating
  (provided $\eps^+ <1$).  Moreover, $(B_\eps(Y),g)$ is an almost
  product in the sense of App.~A.2 in~\cite{post:12}, and the relative
  distortion factor is $\sqrt C$.
\end{proof}
%-----------------------------------------------------------------------

%----------------------------------------------------------------------
% yyyy
%
% Bibliography
%
%----------------------------------------------------------------------

% \bibliographystyle{/home/post/Aktuell/BibTeX/my-amsalpha}
                               % on Olaf's computer
% \bibliography{/home/post/Aktuell/BibTeX/literatur}

%%%%%%%%%%%%%%%%%%%%%%%%%%%% Bibliography %%%%%%%%%%%%%%%%%%

\def\cprime{$'$}
\providecommand{\bysame}{\leavevmode\hbox to3em{\hrulefill}\thinspace}

%\tableofcontents

\end{document}